\documentclass[11pt]{amsart}
\usepackage{hyperref}
\usepackage[]{todonotes}

\hypersetup{
	colorlinks=true,
	linkcolor=blue,
	filecolor=blue,      
	urlcolor=red,
	citecolor=green,
}

\usepackage{bm}
\usepackage{amscd, amsmath, amssymb, amsthm}
\usepackage[all]{xy}
\usepackage{color}

\newtheorem{thm}{Theorem}[section]
\newtheorem{cor}[thm]{Corollary}
\newtheorem{lem}[thm]{Lemma}
\newtheorem{prop}[thm]{Proposition}

\theoremstyle{definition}
\newtheorem{defn}{Definition}[section]
\theoremstyle{remark}
\newtheorem{rem}{Remark}[section]

\numberwithin{equation}{section}

\newcommand{\C}{{\mathbb{C}}}

\newcommand{\R}{{\mathbb{R}}}

\newcommand{\K}{\mathbb{K}}

\newcommand{\Lie}{\mathcal{L}}
\newcommand{\Li}{\mathbb{L}}

\newcommand{\sch}{s^{\nabla}}
\newcommand{\rch}{\mathrm{\rho}^{\nabla}}

\newcommand{\D}{{\rm D}}

\newcommand{\ric}{{\rm Ric}}
\newcommand{\tr}{{\rm tr}}

\newcommand{\en}{{\rm End}}
\newcommand{\aut}{{\rm Aut}}

\newcommand{\di}{{\rm div}}
\newcommand{\Hess}{{\rm Hess}}

\newcommand{\la}{\bm{\langle}}
\newcommand{\ra}{\bm{\rangle}}

\newcommand{\Ham}{{\rm Ham}(M,\omega)}
\newcommand{\ham}{\mathfrak{ham}}
\newcommand{\vol}{{\rm{vol}}}

\newcommand{\thmref}[1]{Theorem~\ref{#1}}

\newcommand{\lemref}[1]{Lemma~\ref{#1}}
\newcommand{\defref}[1]{Definition~\ref{#1}}
\newcommand{\propref}[1]{Proposition~\ref{#1}}
\newcommand{\corref}[1]{Corollary~\ref{#1}}
\newcommand{\remref}[1]{Remark~\ref{#1}}

\newcommand{\ac}{{\mathcal{AC}}_{\omega}}
\newcommand{\g}{{\rm{grad}}}
\newcommand{\End}{\operatorname{End}}
\newcommand{\IM}{\operatorname{Im}}

\usepackage{graphicx}

\begin{document}

\title[complexified orbit] % running head version
{Hermitian Calabi functional in complexified orbits}

%%
% See the complete version of this file, testart.tex, for more
% complicated examples
%
\author{Jie He}
\address[Jie He]{School of Mathematics and Physics, Beijing University of Chemical Technology,  Chaoyang District, Beijing 100029, P.R. China}

%% Note the doubled @@:
\email[Jie He]{hejie@amss.ac.cn; moltke@sina.com}
%\thanks{Research supported in part by NSF grantCCR-87-10433 and DARPA Contract N00019-89-J-1988.}
\author{Kai Zheng}
\address[Kai Zheng]{University of Chinese Academy of Sciences, Beijing 100049, P.R.
China}
%% Note the doubled @@:
\email[KaiZheng]{KaiZheng@amss.ac.cn}
%\date{\today}

%\subjclass{Primary 05C38, 15A15; Secondary 05A15, 15A18}

\maketitle

\begin{abstract}
Let $(M,\omega)$ be a compact symplectic manifold. We denote by $\ac$ the space of all almost complex structure compatible with $\omega$.  $\ac$ has a natural foliation structure with the complexified orbit as leaf. We obtain an explicit formula of the Hessian of Hermitian Calabi functional at an extremal almost K\"ahler metric in $\ac$. We prove that the Hessian of Hermitian Calabi functional is semi-positive definite at critical point when restricted to a complexified orbit, as corollaries we obtain some results analogy to K\"ahler case. We also show weak parabolicity of the Hermitian Calabi flow.
\end{abstract}
\tableofcontents

\section{Introduction}
%Extremal K\"ahler metrics(c.f. \cite{gauduchon2010calabi, tian2000canonical, szekelyhidi2014introduction}), 

%In the 1980's, Calabi \cite{MR645743,MR780039} introduced extremal K\"ahler metrics on K\"ahler %manifolds, which are one of fundamental topics in complex geometry. 
Extremal almost K\"ahler (EAK) metric extends Calabi's extremal K\"ahler metric \cite{MR645743,MR780039} on a symplectic manifold. They are critical points of the Hermitian Calabi functional, which is the squared norm of the Hermitian scalar curvature. Hermitian Calabi flow is the gradient flow of the Hermitian Calabi functional.

In this paper, we compute Hessian of the Hermitian Calabi functional and prove weak parabolicity of the Hermitian Calabi flow.

%Extremal K\"ahler metrics, which was introduced by Calabi \cite{MR645743,MR780039} as an attempt to seek canonical metrics in K\"ahler manifolds in the 1980's, are one of fundamental topics in complex geometry.
%Extremal K\"ahler metrics are critical points of the $L^2$ norm of scalar curvature of K\"ahler metrics (known as the Calabi functional)within a given K\"ahler class.  

%When the ambient manifolds are almost K\"ahler, the notaion of extremal K\"ahler metric has a natural generalisation(\cite{apostolov2003curvature}, \cite{lejmi2010extremal},\cite{gauduchon2010calabi}). 

%Extremal K\"ahler metrics were extended in the almost K\"ahler category. 
Before we state results explicitly, we recall some notions.
We let $(M,\omega)$ be a symplectic manifold, where $\omega$ is a given symplectic form on $M$.
An almost K\"ahler structure consists a symplectic manifold $(M,\omega)$ and an $\omega$-compatible almost-complex structure $J$, which means they satisfy two conditions
\begin{align*}
\text{$\omega$-tamed: }& \omega(\cdot,J\cdot)>0;\\
J\text{-invariant: }& \omega(J\cdot,J\cdot)=\omega(\cdot,\cdot).
\end{align*}
The compatibility conditions leads to a $J$-invariant Riemannian metric 
\begin{align}\label{compatibility}
g_J(\cdot,\cdot):=\omega(\cdot , J\cdot).
\end{align}
From now on, we always fix the symplectic form $\omega$ and we use the almost complex structure $J$ and the Riemannian metric $g$ interchangeable.

We collect all $\omega$-compatible almost-complex structures in the set
\begin{align*}
\mathcal{AC}_{\omega}:=\{J\in \End(TM):J^2=-1 \text{ and $J$ satisfies \eqref{compatibility}} \},
\end{align*}
which is an infinite-dimension K\"ahler manifold equipped with the complex structure $\mathbb J$ and the $L^2$ Riemannian metric 
\begin{align}\label{defj}
\mathbb J(v):=Jv,\quad \la u,v\ra_J:=\int_Mg_J(u,v)\vol, \quad \vol:=\frac{\omega^m}{m!}, \quad  \forall u,v\in T_J\ac.
\end{align}

Let $\sch(J)$ be the Hermitian scalar curvature of an almost-complex structure $J\in\mathcal{AC}_{\omega}$. 
According to Fujiki \cite{MR1207204} and Donaldson \cite{MR1622931}, the Hermitian scalar curvature is the moment map under the action of Hamiltonian symplectomorphisms $\Ham$ on $\ac$.
The Hermitian Calabi functional $\mathcal C:\ac\to \mathbb R$ is interpreted as the squared norm of the moment map, 
$$
\mathcal C (J):=\int_M[\sch(J)]^2\vol.
$$
The extremal almost K\"ahler (EAK) metrics satisfy the equation
\begin{align}
\Lie_{\K}J=0,
\end{align}
where, $\K=:J\g \sch$ is the extremal vector field and $\Lie$ is the Lie derivative. Clearly, almost K\"ahler metric with constant Hermitian scalar curvature is EAK.

Donaldson \cite{MR1622931,MR1736211,MR2103718} initialed a programme on the study of EAK metrics.
There are many explicit examples of EAK (non-K\"ahler) metrics provided in \cite{MR2807093} by Apostolov, Calderbank, Gauduchon and T{\o}nnesen-Friedman.
Recently, many results on extremal K\"ahler metrics have been extended to the EAK metrics. Lejmi \cite{MR2747965} generalised the Futaki invariant and the extremal vector field to the almost K\"ahler setting. Keller-Lejmi \cite{MR4118148} obtained the lower bound of the Hermitian Calabi functional. The deformation problem for the EAK metrics has been studied in \cite{MR2661166, MR1916979}. Vernier \cite{MR4174303} constructed almost K\"ahler metric with constant Hermitian scalar curvature by the gluing method. Legendre \cite{MR3967373} proved that under toric symmetry, the existence of EAK metric implies the existence of extremal K\"ahler metric.
In general, the existence of EAK metrics are expected. We refer interested readers to the survey \cite{MR1969266} of Apostolov and Dr\u{a}ghici on almost K\"ahler geometry.
%There are also many works on Calabi-Yau equation on almost K\"ahler manifolds \cite{MR2330417, MR2787695,MR2313334,MR2743450, MR3959856}. 

%Let $\Ham$ be the group of Hamiltonian symplectomorhisms, which has a natural action on $\ac$. We denote by $\mathcal O_J$ the orbit of $\Ham$ action through $J$. 

The  Hamiltonian symplectomorhisms group $\Ham$ 
has a natural action on $\ac$. We denote by $\mathcal O_J$ the orbit of $\Ham$ action through $J$. 
Any element $\g_{\omega}f$ in the Lie algebra $\ham$ of $\Ham$ defines a vector field on the tangent space of $\mathcal O$:
$$
P(f):=\frac{1}{2} \mathcal L_{\g_{\omega}f}J.
$$
We also define $JP=J\circ P: C^{\infty}(M,\R)\to T_J\ac$ and their adjoint operators $P^\ast$ and $(JP)^*$, see Definition \ref{Past} for an accurate statement.

Thanks to Donaldson's observation \cite{MR1622931}, the Lie algebra $\ham$ could be complexified and it induces a distribution $D$ in the tangent space $T\ac$ as follows
$$
 D_J = \{P(f), JP(f):  \text{ for all }f\in C^{\infty}(M,\R)\},\quad \forall J\in\ac .
$$ 
%where $C_0^{\infty}(M,\R)$ is the set of smooth function with zero average.
Actually, $ D$ is holomorphic and integrable, it generates an integral submanifold $\mathcal D\subset \ac$, which is called a complexified orbit. %(see \cite[Chapter 4]{MR1787650}). 
We denote by $\mathcal D_J$ the complexified orbit through $J$ (we may omit the lower index of $\mathcal D_J$ for convenience).
% since for a given $J$, the complexified orbit is understood). 

Using the notations above, we could explicitly state the Hessian $\Hess \mathcal C$ of the Hermitian Calabi functional.
\begin{thm}\label{prop:second main}
For any $J\in\ac$ and for $u,v\in T_J\ac$,  we choose a two-parameter family of $J(t_1,t_2)$ such that $J(0,0)=J,\quad \partial_{t_1}J(0,0)=u,\quad \partial_{t_2}J(0,0)=v$.
Then we have
\begin{align*}
\Hess \mathcal C(u,v)=-  \la \frac{\partial^2 J}{\partial t_1\partial t_2}\vert_{(0,0)}, J\Lie_\K J\ra
+\la  u,H(v) \ra
-\la u, v\mathcal L_{\K }J\ra.
\end{align*}
Here, we introduce the operator 
\begin{align*}
H(u) := 2JP(JP)^*u-J\Lie_{\K}u.
\end{align*}
Furthermore, we have the following applications.
\begin{enumerate}
\item \label{prop:second eak}
If $J$ is EAK, then we have
\begin{align}\label{hessc0}
\Hess \mathcal C(u,v)=\la H(u),v\ra,
\end{align}
moreover, the operator $J\Lie_{\K}$ is self-adjoint on $T_J\ac$ and semi-positive on $T_J\mathcal D$, see \lemref{jlk}.

\item \label{calabih} 

The EAK metric $J$ is a local minimum of the Hermitian Calabi functional on the complexified orbit $\mathcal D_J$.

If $J$ is EAK, then $\Hess \mathcal C$ restricted to $\mathcal D_J$ is semi-positive 
$$
\Hess\mathcal C(v,v)\geq 0, \ \ \ \forall v\in T_J\mathcal D.
$$ 
Moreover, $\Hess\mathcal C $ is strictly positive on the subspace $\IM JP$,
and vanishes on the subspace $\IM P$. 
Precisely, for any $f_1, f_2\in C^{\infty}(M,\R)$, we have 
\begin{align*}
\Hess\mathcal C(JP(f_1), JP(f_2))
&=2\la \Li(f_1),\Li(f_2)\ra-\frac{1}{2}\la \mathcal L_{\K } (f_1),\mathcal L_{\K }(f_2)\ra\\
&=2\la\Li^+(f_1), \Li^-(f_2)\ra,
\\
\Hess\mathcal C(P(f_1), P(f_2))
&=0.
\end{align*}
In which, $\Li$ is the Lichnerowicz operator $\Li=P^*P$, c.f. \defref{myl1} and the Calabi operators $\Li^{\pm}$ are self-adjoint and semi-positive, c.f. Definition \ref{cala12}.

\item \label{prop:second chsc}
The almost K\"ahler metric of constant Hermitian scalar curvature is a local minimum of the Hermitian Calabi functional on $\ac$.

Actually, if $J$ has constant Hermitian scalar curvature, then
\begin{align} \label{hesschsc}
\Hess \mathcal C(u,v)=2\la  (JP)^*u,(JP)^*v\ra,
\end{align}
which is semi-positive on $T_J\ac$ and vanishes iff $v\in\ker (JP)^*$.

\item \label{prop:second geodesic}
If $J(t)$ is a geodesic in terms of the Riemannian metric \eqref{defj} in $\ac$, then the geodesic equation satisfies $J''=JJ'J'$ and the second order derivative of the Hermitian Calabi functional along $J(t)$ obeys
\begin{align} \label{hess geodesic}
\frac{d^2}{dt^2}\mathcal C(J_t)=
\la H(J'), J'\ra.
\end{align}
\end{enumerate}
\end{thm}

%\begin{align}\label{ctang}
%T_J\mathcal D = \IM P+\IM J P.
%\end{align} 
By \eqref{calabih} in \thmref{prop:second main}, we can get a structure property of the tangent space of the complexified orbit.
\begin{cor}\label{cor1}
If $J$ is EAK, then $$\IM P\cap \IM JP=\{0\},\quad T_J\mathcal D_J=\IM P\oplus \IM JP.$$
\end{cor}

\begin{cor}\label{invar}
On an almost K\"ahler manifold $(M,\omega, J)$,  
\begin{itemize}
\item
The Hermitian Calabi functional is invariant under the action of  $\Ham$.
\item If we restrict the Hermitian Calabi functional $\mathcal C$ to a complexified orbit $\mathcal D$, $J$ is a critical point of $\mathcal C$ iff it is a local minimum of $\mathcal C$.
\item
The space of EAK metrics in $\mathcal D_J$ is a submanifold whose each connected component is an orbit of the Hamiltonian group $\Ham$.
\end{itemize}
\end{cor}

\begin{rem}
Due to \eqref{prop:second eak} in \thmref{prop:second main}, it could be possible that the EAK metric becomes a saddle point of the Hermitian Calabi functional on $\ac$. While, it depends on the eigenvalues of the operators $2JP(JP)^*$ and $J\Lie_{\K}$. If so, it would be interesting to search the saddle point by applying the minimax approach
\begin{align*}
\max_{\tilde J\in \ac} \min_{J\in \mathcal D_{\tilde J}} \mathcal C(J).
\end{align*}
\end{rem}

\begin{rem}
When the background manifold is K\"ahler, a complexified orbit can be identified with a K\"ahler class via Moser's lemma. Thus \eqref{calabih} in \thmref{prop:second main} is a generalisation of Calabi's classical results \cite[Theorem 2]{MR780039}. That is the Hessian of the Calabi functional is semi-positive at an extremal K\"ahler metric. Our proof of \eqref{calabih} in \thmref{prop:second main} applies substantial results in Gauduchon's book \cite{gauduchon2010calabi}, where he used Mohsen formula to characterise the EAK condition, made a very detailed study of $\ac$ and the Hermitian Calabi functional. Our \thmref{prop:second main} follows his work to compute the Hessian of the Hermitian Calabi functional.
\end{rem}

\begin{rem}
In \cite{MR2249564}, Lijing Wang gave a different proof of Calabi's result \cite{MR780039} on the Hessian of the Calabi functional. His method is based on reductive group action admitting a moment map on a K\"ahler manifold. %\thmref{prop:second main} extends this results to the almost K\"ahler setting.
%However, his proof cannot apply to our case since he required the group is reductive holomorphic and the background manifold is K\"ahler. \
\end{rem}

\begin{rem}\label{salamonstruc}
If  the Hermitian scalar curvature metric $\sch_J$ of $J$ is constant, Garc\'ia-Prada and Salamon (\cite[Remark 2.10]{MR4234988}, \cite[Corollary 1.12]{MR4285699}) showed that  $\IM P\perp \IM JP$, so $  T_J\mathcal D_J=\IM P\oplus \IM JP$. \corref{cor1} is a generalisation of their result in the EAK case.
\end{rem}

\begin{rem}
Calabi \cite{MR780039} proved that the Hessian of Calabi functional is strictly positive along the directions transversal to the identity component of automorphism group $\aut_0(M,J)$ orbit. \corref{invar} is an analogue to Calabi's result, while the Hamiltonian group plays the role as automorphism group in the K\"ahler case.
\end{rem}

In Section \ref{Hermitian Calabi flow}, we study the Hemritian Calabi flow. 
 The Hermitian Calabi flow has appeared in the convergence problem of the Calabi flow \cite{MR2103718,MR3969453} and uniqueness of the adjacent constant scalar K\"ahler metrics \cite{MR3224716}.

 The Hermitian Calabi flow is the negative gradient flow of the Hermitian Calabi functional
$$
\frac{d}{dt}J = \frac{1}{2}J\Lie_\K J.
$$ 
Alternatively, we have $\frac{d}{dt}J =JP(\sch(J))$, which suggests that the Hermitian Calabi flow $J(t)$ would stay in the distribution $\mathcal D$ as long as it exists.

\begin{thm}
The Hermitian Calabi flow 
is a 4th order weakly parabolic system. For any $\xi\in T^*M$, the principal symbol of its linearisation is given by
$$
\hat\sigma_4(x,\xi)v  = \frac{1}{2}(v,\Xi)\Xi,
$$
where $\Xi\in T_J\ac$ is given by $\Xi = \xi^{\sharp}\otimes (J\xi)+ (J\xi^{\sharp})\otimes \xi$.
\end{thm}

In Appendix \ref{sect}, we compute an explicit expression of the Lichnerowicz operator $\Li$, which appears in \eqref{calabih} in \thmref{prop:second main}.
\begin{thm}\label{5}
On an almost K\"ahler manifold $(M,\omega, J)$, the Lichnerowicz operator has the following explicit expression:
\begin{align}\label{4}
\Li(f)=&\frac{1}{2}\Delta^2f-2(\delta\ric^+, df)+2(\rho,dd^c f)+\delta\delta(\D^{+}df-\D^{-}df)
\end{align}
for all $f\in C^{\infty}(M,\mathbb R)$. Where $\delta$ is the formal adjoint of Levi-Civita connection $\D$ and  $\D^{\pm}\alpha$ is the $J$-invariant(resp J-anti-invariant) part of $\D\alpha$, i.e.
\begin{align*}
\D^{\pm}\alpha(X,Y)=\frac{1}{2}[(\D_X\alpha) (Y)\pm (\D_{JX}\alpha)(JX)],
\end{align*}
$\ric^+$ is the $J$-invariant part of Ricci curvature and $\rho(X,Y)=\ric^+(JX,Y)$.
\end{thm}
\begin{rem}Our result \eqref{4} in \thmref{5} is a continued calculation of Vernier's formula \cite[Equation (11)]{MR4174303}, where the Lichnerowicz operator is a 4th order elliptic operator plus an error term. We write the error term in an explicit way.
\end{rem}

\begin{rem}
On K\"ahler manifolds, we have $$\ric=\ric^+,\quad \delta(\D^{+}df-\D^{-}df)(X)=\ric(\g f, X).$$ So the formula in \eqref{4} becomes
$$
\Li(f)=\frac{1}{2}\Delta^2f-(\delta\ric^+, df)+(\rho,dd^c f)=\frac{1}{2}\Delta^2f+\frac{1}{2}( ds, df)+(\rho, dd^c f),
$$ 
which is exactly Gauduchon's real Lichnerowicz operator \cite[Lemma 1.23.5]{gauduchon2010calabi}, and twice of the real part of Calabi's Lichnerowicz operator \cite[Proof of Theorem 2]{MR780039}, see also \cite{MR3010175,gauduchon2010calabi}. 
\end{rem}

\subsection*{Acknowledgements}
%Jie He thanks Kai Zheng for conversations and helpful suggestions, he also 
Jie He would like to thank Youde Wang for his help and support.
K. Zheng is partially supported by NSFC grant No. 12171365.

\section{Preliminaries}
In this section, we will introduce the basic materials in almost K\"ahler geometry. From now on, $(M,\omega)$ is always a compact symplectic manifold. Let $J$ be an $\omega$-compatible almost complex structure and $g$ be the Riemannian metric determined by (\ref{compatibility}).

We choose a local orthogonal frame of $TM$ as
\begin{align}\label{basis}
	\{e_1,e_2,\ldots,e_m,Je_1=e_{m+1},Je_2=e_{m+2},\ldots,Je_m=e_n\}
\end{align}with the dual frame 
\begin{align*}
	\{e_1^*,e_2^*,\ldots,e_m^*,Je_1^*=e_{m+1}^*,Je_2^*=e_{m+2}^*,\ldots,Je_m^*=e_n^*\}
\end{align*}
such that \begin{align*}
	\omega=\frac{1}{2}\sum_{i=1}^ne_i^*\wedge Je_i^*.
\end{align*}

The almost complex structure is extended to any $p$-form $\psi\in\Omega^p(M)$ as
\begin{align*}
J\psi(X_1, \ldots, X_p) := \psi(J^{-1}X_1, \ldots, J^{-1}X_p),
\end{align*}
for all $X_1, \ldots, X_p\in\Gamma(TM)$, where $J^{-1}:=-J$. We define $Jf:=f$ if $f$ is a  function.

The twisted differential operator $d^c$ and the twisted codifferential operator $\delta^c$ are defined by changing the differential operator $d$ and the codifferential operator $\delta$ under the extended almost complex structure $J$
\begin{align*}
d^c:=J  d  J^{-1},\quad \delta^c:=J  \delta  J^{-1}.
\end{align*}

We denote by $\la,\ra$ the $L^2$ inner product of any tensor of type $(p,q)$ over $M$, and by $(, )$ the inner product on fibre induced by $g$, i.e.
$$
\la S,T\ra=\int_M(S,T)\vol=\int_Mg(S,T)\vol,\quad \forall S,T\in (TM)^{\otimes p}\otimes (T^*M)^{\otimes q}.
$$

%Denote by $\D$ the Levi-Civita connection of $g$ and $\D^*$ its formal adjoint.
%For any $(p,q)$-tensor $T$ with $q\geq 1$, $\D^* T$ is a $(p,q-1)$ tensor and (see \cite[1.55]{MR2371700}), 
%\begin{align}\label{delta}
%\D^\ast T=-\tr(\D T)=- \sum_i^n e_i\lrcorner \D_{e_i}T.
%\end{align}
%But for any $p$-form $\psi$(see \cite[1.56]{MR2371700}),we also have
%\begin{align}
%\delta\psi := (-1)^{np+n+1}*d*\psi=- \sum_i^n e_i\lrcorner \D_{e_i}\psi=\D^\ast \psi.
%\end{align}
%By an abuse of notation, we use $\delta$ to denote both the codifferential operator and the formal adjoint $\D^*$ of the Levi-Civita derivative $\D$ since they have the same expression. 
%\begin{align}\label{delta-def}
%\la \delta T,S\ra=\la T,\D S\ra, \quad \text{for any $(p,q)$-tensor $T$ and $(p,q-1)$-tensor $S$}
%\end{align}

Denote by $\D$ the Levi-Civita connection of $g$ and $\delta$ its formal adjoint. We have (see \cite[1.55]{MR2371700}), 
\begin{align}\label{delta}
\delta T=-\tr(\D T)=- \sum_i^n e_i\lrcorner \D_{e_i}T.
\end{align}
For any $(p,q)$-tensor $T$ and $(p,q-1)$-tensor $S$ with $q\geq1$, the adjointness implies
\begin{align}\label{delta-def}
\la \delta T,S\ra=\la T,\D S\ra.
\end{align}

For any $J\in\ac$, the Nijenhuis tensor $N$ of $J$ satisfies the formula
$$
N(X,Y)=\frac{1}{4}([JX,JY]-J[JX,Y]-J[X,JY]-[X,Y]).
$$
By its very definition, we have \begin{align}\label{jantiinva}
N(JX,Y)=N(X,JY)=-JN(X,Y).
\end{align}

The symplectic condition $d\omega=0$ yields a relation between $\D J$ and $N$.
\begin{lem}[{\cite[Lemma 9.3.1]{gauduchon2010calabi},\cite[Proposition 4.2]{MR1393941}}]\label{djandn}
On any almost K\"ahler manifold $(M,\omega, J)$, it holds
\begin{align}\label{29}
 ((\D_XJ)Y,Z)=2(JX,N(Y,Z)), \quad \forall X,Y,Z\in TM.
\end{align}
\end{lem}

A direction consequence of (\ref{jantiinva}) and (\ref{29})  is the following corollary.
\begin{cor}
Let $(M,\omega, J)$ be an almost K\"ahler manifold. We have
\begin{align}\label{23}
\D_{JX}J=(\D_XJ)J,\quad \forall X\in TM.
\end{align}
\end{cor}
\begin{proof}
It follows from compatible condition and (\ref{29}) that
\begin{align*}
((\D_{JX}J)Y,Z)=2(JJX,N(Y,Z))=-2(JX,JN(Y,Z)).
\end{align*}
It follows from (\ref{jantiinva}) and (\ref{29}) that
\begin{align*}
((\D_XJ)JY,Z)=2(JX,N(JY,Z))=-2(JX,JN(Y,Z)).
\end{align*}
Thus these two identities establish the relation \eqref{23}.
\end{proof}

\begin{cor}\label{identity}
Let $(M,\omega, J)$ be an almost K\"ahler manifold. It holds
$\delta J = 0$, where $\delta$ is the formal adjoint of $\D^*$ as we defined in (\ref{delta}).
\end{cor}
\begin{proof}
It follows from  (\ref{23}) that
\begin{align*}
\delta J = -\sum_{i=1}^n(\D_{e_i}J)e_i=-\sum_{i=1}^n(\D_{Je_i}J)Je_i=\sum_{i=1}^n(\D_{e_i}J)e_i=-\delta J.
\end{align*}
Consequently, $\delta J=0$.
\end{proof}

\begin{lem}\label{lem36}
On an almost K\"ahler manifold $(M,\omega, J)$, it holds
\begin{align*}
\delta  d^cf=0, \quad \forall f\in C^{\infty}(M).
\end{align*}
\end{lem}
\begin{proof}
With the help of the definition of $d^c$, we see $d^cf$ is a $1$-form. Direct computation shows
\begin{align}
\label{23 new}
\begin{split}
\delta  d^cf=&-\sum_{i=1}^n(\D_{e_i}d^cf)(e_i)=\sum_{i=1}^n-\D_{e_i}(d^cf(e_i))+d^cf(\D_{e_i}e_i)\\
=&\sum_{i=1}^n\D_{e_i}(df(Je_i))-df(\sum_{i=1}^nJ\D_{e_i}e_i).
\end{split}
\end{align}
It follows from \corref{identity} that
$$
\delta J = -\sum_{i=1}^n(\D_{e_i}J)(e_i)=-\sum_{i=1}^n\D_{e_i}(Je_i)+\sum_{i=1}^nJ\D_{e_i}e_i=0,
$$
i.e.
\begin{align}\label{24}
\sum_{i=1}^n\D_{e_i}(Je_i)=\sum_{i=1}^nJ\D_{e_i}e_i.
\end{align}
Substituting (\ref{24}) into (\ref{23 new}) yields
\begin{align*}
\delta  d^cf=\sum_{i=1}^n(\D_{e_i}(df(Je_i))-df(\D_{e_i}(Je_i))=\sum_{i=1}^n\Hess f(e_i, Je_i).
\end{align*}
Denote $\Hess f(e_i,e_j)=\Hess_{ij}, Je_i=J^j_ie_j$. Then we have $\Hess_{ij}=\Hess_{ji}, J_i^j=-J^i_j$. So we get
$$
\sum_{i=1}^n\Hess f(e_i, Je_i)=\sum_{i,j}\Hess_{ij}J^j_i=-\sum_{i,j}\Hess_{ji}J^i_j=-\sum_{i=1}^n\Hess f(e_i, Je_i),
$$
i.e., $\delta  d^cf=\sum_{i=1}^n\Hess f(e_i, Je_i)=0$.
\end{proof}

\subsection{Hamitonian group}
Denote $\Ham$ the symplectic Hamiltonian group of $(M,\omega)$, and $\ham$ the corresponding Lie algebra, then
$$
\ham = \{X\in \Gamma(TM):\iota_X\omega\text{ is exact}\}.
$$
Writing $\iota_{X}\omega=-df$, the function $f$ is called the momentum of $X$ regarding to $\omega$, and $$X=\g_{\omega}f:=J\g f,$$ where $\g f$ is the Riemannian gradient of $f$. We call $\g_{\omega}f$ the \emph{symplectic gradient} of $f$.

A Hamiltonian vector field has many momentums which may differ by a constant. If we convent that the integral of the momenta function is 0, then it is unique. Under the corresponding $$X=\g_{\omega}f\to f,$$  $\ham$ is identified to the set of  smooth function with zero average which is denoted by $\ham_{\omega}$,
$$
\ham_{\omega} = \{f\in C^{\infty}(M,\R):\int_Mf\omega^m=0\}:=C_0^{\infty}(M).
$$
This identification is in fact a Lie algebra identification: if we define the Poisson bracket over $C^{\infty}_0(M,\mathbb R)$
\begin{align}
\{f, g\}=\omega(\g_{\omega}f, \g_{\omega}g) = \g_{\omega}f(g)=-\g_{\omega}g (f),
\end{align}
then 
$$
[\g_{\omega}f, \g_{\omega}g]=\g_{\omega}\{f, g\}\to \{f, g\}.
$$
where $[, ]$ is the Lie bracket of vector fields.

\begin{defn}
On an almost K\"ahler manifold $(M,\omega, J)$, we say a real vector field $X$ is holomorphic if $$\Lie_XJ=0.$$
\end{defn}
The real holomorphic vector fields constitutes a Lie subalgebra under the Lie bracket of vector fields.
%%%%%%%%%%%%%%%%%%%%%%%%%%%%%%%%%%%%%%%%%%%%%%%%%%%%%%%%%%%%%%%%%%%%%%%%%%%%%%%%%%%%%%%%%%%%%%%%%%%%%%%%%%%%%%%%%%%%%%%%%%%%%%%%
\subsection{Hermitian scalar curvature}
On an almost K\"ahler manifold $(M,\omega, J)$, the Hermitian connection $\nabla$(c.f.\cite{MR1456265}) is defined by 
$$
\nabla_XY = \D_XY-\frac{1}{2}J(\D_XJ)Y.
$$
When $(M,\omega,J)$ is K\"ahler,  $\D,\nabla$ coincide with each other. 

% Instead  they are related by $\nabla$(c.f. \cite{MR1969266, fu2022scalar, MR1456265}) 

We denote $R$ the Levi-Civita curvature tensor and  $R^{\nabla}$ the canonical Hermitian  curvature tensor, i.e.
$$
R(X,Y)=\D_{[X,Y]}-[\D_X, \D_Y],\ \  R^{\nabla}(X,Y)=\nabla_{[X,Y]}-[\nabla_X, \nabla_Y].
$$

We denote by $\ric$ the  Riemann Ricci curvature and  by $\ric^+$ the $J$-invariant part of $\ric$, 
\begin{align}\label{ric+}
\ric^+(X,Y)=\frac{1}{2}(\ric(X,Y)+\ric(JX,JY)).
\end{align}
The compatibility of $\ric^+$ and $J$ determines a 2-form $\rho$  
\begin{align}\label{rho}
\rho(X,Y)=\ric^+(JX,Y).
\end{align}
We define another 2-form $\rho^*$ via contracting $R$ in terms of $\omega$, 
\begin{align}\label{rho*}
\rho^*(X,Y)=\frac{1}{2}\sum_{i=1}^ng(R(X,Y)e_i, Je_i).
\end{align}
The (0,2)-tensor $\ric^+, \rho, \rho^*$ will be used in the computation of Lichnerowicz operator in Section \ref{sect}. The Hermitian Ricci form is the contraction of $R^{\nabla}$ by $\omega$, 
\begin{align}
\rch_J(X,Y) = \frac{1}{2}\sum_{i=1}^ng(R^{\nabla}(X,Y)e_i,Je_i).
\end{align}

The Hermitian scalar curvature is defined by
\begin{align}
\sch_J:=\sch(J) = 2(\rch_J,\omega).
\end{align}
The averaged Hermitian scalar curvature 
\begin{align*}
\underline s^{\nabla}=\int_M \sch_J \vol/\int_M\vol
\end{align*} is a topological constant which does not depend on the $J\in\ac$.

We denote by $\K$ the symplectic gradient of the Hermitian scalar curvature, i.e.
\begin{align}
\K=\g_{\omega}\sch_J.
\end{align}
When $J$ is EAK, $\K$ is exactly the \emph{extremal vector field} (EVF). In general, EVF is defined to be $\g_{\omega} (\Pi_\omega\sch_J)$, where $\Pi$ is the $L^2$-orthogonal projection in $\ham_{\omega}$ and $\Pi_\omega\sch_J$ is independent of $J$. EVF was first introduced by Mabuchi and Futaki \cite{MR1314584} in K\"ahler geometry.  Lejmi \cite[Section 3.2]{MR2747965} generalised this notion to almsot K\"ahler manifolds. $\K$ is very important in the variation of Calabi functional.

%%%%%%%%%%%%%%%%%%%%%%%%%%%%%%%%%%%%%%%%%%%%%%%%%%%%%%%%%%%%%%%%%%%%%%%%%%%%%%%%%%%%%%%%%%%%%%%%%%%%%%%%%%%%%%%%%%%%%%%%%%%%%%%%%%%%%%
%\subsection{The space of almost complex structures compatible with a symplectic form}

\subsection{Complexified orbit in $\ac$}

Recall that $\ac$ consists of all $\omega$-compatible almost complex structures, and its tangent space is 
$$
T_J\ac=\{v\in\en(TM):vJ+Jv=0, \omega(JX,vY)+\omega(vX,JY)=0\}.
$$
By the compatible condition, $\omega(JX,vY)+\omega(vX,JY)=0$ is equivalent to $(X,vY)=(vX,Y)$, i.e., $v$ is self-adjoint.
So we have an equivalent characterisation of $T_J\ac$,
\begin{align}\label{tangent}
T_J\ac=\{v\in\Gamma(\End(TM)): vJ+Jv=0, \  (vX,Y)= (X,vY)\}.
\end{align}

The Hamiltonian group $\Ham$ has a natural action on $\ac$ by
$$
(\phi, J)\to \phi_*J\phi^{-1}_*,\quad  \forall \phi\in\Ham, \quad J\in\ac.
$$
The tangent space of the resulting orbit $\mathcal O = \{\phi_*J\phi^{-1}_*:\phi\in\Ham\}$ through $J$ is
$$
 T_J\mathcal O =\{\mathcal L_{X}J:X\in\ham \}=\{\mathcal L_{\g_{\omega}f}J:f\in C^{\infty}_0(M) \}.
$$

\begin{lem}\label{lem25}
For any $f\in C^{\infty}(M,\R)$, 
$$
{\mathcal L}_{\g_{\omega}f}J,\ \  J{\mathcal L}_{\g_{\omega}f}J\in T_J\ac.
$$
\end{lem}
\begin{proof}
By (\ref{tangent}) we need to prove that ${\mathcal L}_{\g_{\omega}f}J,\ J{\mathcal L}_{\g_{\omega}f}J $ satisfy 
\begin{enumerate}
\item \label{condi1}
$vJ+Jv=0$;  
\item\label{condi2}
$ (vX,Y)= (X,vY),\ \  \forall X,Y\in \Gamma(TM).$
\end{enumerate}
Taking Lie derivative on both sides of $J^2=-1$ gives us
\begin{align*}
{\mathcal L}_{\g_{\omega}f}(J^2)={\mathcal L}_{\g_{\omega}f}JJ+J{\mathcal L}_{\g_{\omega}f}J=0,
\end{align*}
i.e. ${\mathcal L}_{\g_{\omega}f}J$ satisfies condition (\ref{condi1}). For $J{\mathcal L}_{\g_{\omega}f}J$, we compute
$$
J(J{\mathcal L}_{\g_{\omega}f}J) =  -J({\mathcal L}_{\g_{\omega}f}JJ) = -(J{\mathcal L}_{\g_{\omega}f}J)J,
$$
i.e. $J{\mathcal L}_{\g_{\omega}f}J$ satisfies condition (\ref{condi1}).

For the second condition, since $\omega(J.,.) = -g(.,.)$ and $\mathcal L_{\g_{\omega}f}\omega=0$, we have
\begin{align}\label{213}
(\mathcal L_{\g_{\omega}f}g )(X,Y)=-\omega({\mathcal L}_{\g_{\omega}f}J X,Y)=-g (J{\mathcal L}_{\g_{\omega}f}J X,Y).
\end{align}
That $g$ is symmetric implies that $\mathcal L_{\g_{\omega}f}g $ is also symmetric. Thus $J{\mathcal L}_{\g_{\omega}f}J$ is self-adjoint and $J{\mathcal L}_{\g_{\omega}f}J$ satisfies condition (\ref{condi2}). For ${\mathcal L}_{\g_{\omega}f}J$, we compute
\begin{align*}
 g({\mathcal L}_{\g_{\omega}f}JX,Y)=& g(J{\mathcal L}_{\g_{\omega}f}JX,JY)\\
=& g(X,J{\mathcal L}_{\g_{\omega}f}JJY)\\
=& g(X,{\mathcal L}_{\g_{\omega}f}JY).
\end{align*}

\end{proof}

% and \cite[Section 2.2]{MR2303522}. 
%%%%%%%%%%%%%%%%%%%%%%%%%%%%%%%%%%%%%%%%%%%%%%%%%%%%%%%%%%%%%%%%%%%%%%%%%%%%%%%%%%%%%%%%%%%%%%%%

According to the Hamiltonian action on $\ac$, any function $f\in C^{\infty}_0(M,\R)=\ham_{\omega}$ induces a tangent vector on $T_J\ac$ by
$$
f\to \Lie_{\g_{\omega}f}J.
$$
The Lie algebra $\ham_{\omega}$ is complexified by using $C_0^\infty(M,\mathbb C)=\ham_{\omega}+\sqrt{-1}\ham_{\omega}$. The imaginary part $\sqrt{-1}f$ in the complexified Lie algebra induces a tangent vector in $T_J\ac$ 
$$
 J{\mathcal L}_{\g_{\omega}f}J.
$$

By \lemref{lem25},  there exists a distribution  $D$ on $T\ac$ given by 
$$
D_J = \{{\mathcal L}_{\g_{\omega}f}J,\ \ J{\mathcal L}_{\g_{\omega}f}J:f\in\ham_{\omega}\},
$$
which can be viewed as the distribution induced by the complexified Lie algebra.

It is obvious that $D_J$ is a holomorphic distribution,  that is 
$$
\mathbb J D_J=D_J.
$$
In 1983, Donaldson\cite[Page 408]{MR1622931} first observed that 
\begin{lem}\label{integrable distribution}
$D_J$ forms an integrable distribution on $\ac$.
\end{lem} We denote by $\mathcal D$ the integral submanifold generated by $D_J$.
For any $J\in\ac$, we call ${\mathcal D}_J$ the complexified orbit through $J$. $\ac$ has a natural foliation structure with $\mathcal D$ as leaf. 

%We remark that later in 2000, Tian \cite[Proposition 4.3]{MR1787650} assumed $J$ is K\"ahler and gave a detailed proof of \lemref{integrable distribution}.
%The tangent space of $\mathcal D$ and $\mathcal O$ is given in (\ref{ctang}).

%%%%%%%%%%%%%%%%%%%%%%%%%%%%%%%%%%%%%%%%%%%%%%%%%%%%%%%%%%%%%%%%%%%%%%%%%%%%%%%%%%%%%%%%%%%%%%%%%%%%%%%%%%%%%%%%

\section{Operators}
In this section, we will introduce some operators related to $\ac$.

\subsection{Operators $P$ and $P^\ast$}

\begin{defn}\label{P}
The operator $P: C^{\infty}(M,\R)\to T_J\ac$ is defined by
$$
P(f)=\frac{1}{2}\mathcal L_{\g_{\omega}f}J.
$$
\end{defn}
\begin{rem}
Comparing with the original definition of  $P$ by Donaldson \cite{MR1622931}(see also \cite[Page 49]{MR1787650}, \cite[Page 6]{MR2663648}), we add a normalisation factor $1/2$ in our definition of $P$. The normalisation factor ensures $P$ and the following $P$ related operations are all natural generalisation of their K\"ahler counterparts.
\end{rem}
In fact $P(f)\in T_J\mathcal O $. We  define $JP=J\circ P$, then $$JP(f)=\frac{1}{2}J\mathcal L_{\g_{\omega}f}J\in  \mathbb JT_J\mathcal O ,$$ 
where $\mathbb J$ is defined in (\ref{defj}).
So we have the decomposition
\begin{align*}
 T_J\mathcal D_J=T_J\mathcal O + \mathbb J  T_J\mathcal O=\IM P+\IM JP.
\end{align*} 

\begin{defn}\label{Past}
	We define $P^\ast:\Gamma(\End(TM))\rightarrow C^{\infty}(M,\R)$ the formal adjoint operator of $P$ via $L^2$ integral. That is, $P^\ast$ satisfies 
	\begin{align*}
		\la P^*(v), f \ra:=\la v, P(f)\ra,\quad  \forall f\in C^{\infty}(M,\R),\quad v\in\Gamma(\End(TM)),
	\end{align*}
	under the $L^2$-inner product $\la\cdot,\cdot\ra$ over $M$ induced by $g$ on $C^{\infty}(M,\R)$ and $\End(TM)$.
\end{defn}
We also define $(JP)^*$ the formal adjoint of $JP$, then by definition, we have
\begin{align*}
\la (JP)^*v, f \ra=\la v, JPf \ra=-\la Jv, Pf\ra=-\la P^*Jv, f\ra,
\end{align*}
i.e. 
\begin{align} \label{def:jpstar}
(JP)^*=-P^*Jv.
\end{align}

The following Lemma is important in the description of variation of Hermitian Calabi functional and moment map.  
It is contained in the proof of (9.6.5) in \cite[Theorem 9.6.1]{gauduchon2010calabi}. We collect it here and reformulate the proof.
\begin{lem} \label{pstar}
For any $v\in T_J\ac$, we have
\begin{align}
P^* v = \delta J(\delta Jv)^{\flat},
\end{align}
where $\delta$ is defined in (\ref{delta}).
\end{lem}
\begin{proof}
For any vector fields $X,Y$ and symplectic vector field $Z$, it follows from the compatible condition $g(X,Y)=\omega(X,JY)$ and $\mathcal L_Z\omega=0$ that
\begin{align*}
(\mathcal L_Zg)(X,Y)=(\mathcal L_Z\omega)(X,JY)+\omega(X,(\mathcal L_ZJ)Y)=\omega(X,(\mathcal L_ZJ)Y).
\end{align*}
Using the formula 
\begin{align*}
(\mathcal L_Zg)(X,Y)= g(\D_XZ,Y)+ g(X, \D_YZ),
\end{align*}
and the compatible condition $\omega(.,.)=g(J., .)$, 
we have
\begin{align}\label{formu26}
g(\D_XZ,Y)+ g(X, \D_YZ)=g(JX,(\mathcal L_ZJ)Y).
\end{align}

In order to compute $\la v,P(f)\ra$, we choose an orthonormal basis as we did in (\ref{basis}). We let $Z=\g_{\omega}f,\ Y=e_i,\ X=-Jv(e_i)$ in (\ref{formu26}) and compute
\begin{align}
\label{basic1}
\begin{split}
\la \mathcal L_{\g_{\omega}f}J, v\ra=&-\int_M\sum_{i=1}^ng((\mathcal L_{\g_{\omega}f}J)e_i, JJv(e_i))
\\
=&-\int_M\sum_{i=1}^n g(\D_{Jv(e_i)}\g_{\omega}f,e_i)+ g(Jv(e_i), \D_{e_i}\g_{\omega}f)\\
=&-\int_M \tr(\D \g_{\omega}f\circ Jv)+g(Jv, \D\g_{\omega}f )\\
=&-\int_M \tr(Jv\circ\D \g_{\omega}f  )+g(Jv, \D\g_{\omega}f ).
\end{split}
\end{align}
Here we view $\D \g_{\omega}f\in\End(TM)$ as $X\to \D_X \g_{\omega}f$. Since $Jv\in T_J\ac$,  $Jv$ is self-adjoint by (\ref{tangent}) and 
\begin{align*}
\tr( Jv \circ\D\g_{\omega}f ) =\sum_{i=1}^n g( Jv\D_{e_i}\g_{\omega}f,e_i)
=\sum_{i=1}^n g( \D_{e_i}\g_{\omega}f,Jve_i)=g(\D\g_{\omega}f, Jv  ).
\end{align*}
It follows from (\ref{basic1}) that
\begin{align}\label{equ34}
\la \mathcal L_{\g_{\omega}f}J, v\ra=-2\la \D\g_{\omega}f, Jv \ra.
\end{align}

Then we compute
\begin{align*}
\la P f, v\ra=&\frac{1}{2}\la \mathcal L_{\g_{\omega}f}J, v\ra
=-\la \D \g_{\omega}f, Jv \ra
=-\la  \g_{\omega}f, \delta Jv \ra \\
=& \la  df, J(\delta Jv)^{\flat} \ra=\la f,\delta J(\delta Jv)^{\flat}\ra.
\end{align*}
Taking adjoint, we thus obtain $P^* v = \delta J(\delta Jv)^{\flat}$.
\end{proof}

We introduce the
Mohsen Formula \cite{mohsen2003symplectomorphismes}(also see \cite[Theorem 2.6]{MR4234988}),
\begin{lem}[Moshsen Formula]\label{mohsen}
For any $v\in T_J\ac$ and any curve $J(t)\in \ac$ satisfying $J(0)=J, J'(0)=v$, the first variation of the Hermtian ricci form and the Hermitian scalar curvature is
\begin{align}
\frac{d}{dt}|_{t=0}\rho^{\nabla}_{J(t)}=-\frac{1}{2}d(\delta v)^{\flat},\quad
\frac{d}{dt}|_{t=0}\sch_{J(t)}=-\delta J(\delta v)^{\flat}.
\end{align}
\end{lem}
If we view Hermitian scalar curvature as functional on $\ac$, combining \lemref{mohsen} and \lemref{pstar} we have
 \begin{align}\label{scalarvaria}
D\sch_J(v)=P^* Jv=-(JP)^*v.
\end{align}
We can immediately obtain the description of EAK condition, i.e. the Euler-Lagrange equation of the Hermitian Calabi functional
$$
\mathcal C(J)=\int_M(\sch(J))^2 \vol.
$$
\begin{cor}[\cite{MR1969266, gauduchon2010calabi}]\label{1stvaria}
$J\in \ac$ is EAK iff $\K=\g_{\omega}s^{\nabla}_J$ is a real holomorphic vector field.
\end{cor}
\begin{proof}
It follows from (\ref{scalarvaria}) that
\begin{align}
\label{ca1}
D\mathcal C(v)
=&-2\la (JP)^* v, s^{\nabla}\ra=-2\la v, JP (\sch)\ra
=-\la v , J\Lie_{\K}J\ra.
\end{align}
Thus $J\in\ac$ is a critical point of $\mathcal{C}$ iff $\Lie_{\K}J=0$, i.e. $\g_{\omega} s^{\nabla}$ is a holomorphic vector field.
\end{proof}

Now we can see \eqref{scalarvaria} yields Donaldson's famous results: the Hermitian scalar curvature is a moment map $\mu:\ac\to\ham^*_{\omega}$ for the Hamiltonian action on $\ac$ via the $L^2$-product
\begin{align*}
	\mu(J)(f) = \int_Mf(\sch_J-\bar s^{\nabla})\vol, \quad \forall f\in \ham_{\omega}.
\end{align*}
For any $f\in \ham$, the induced vector on $T_J\ac$ is $\rho(f) = \Lie_{\g_{\omega}f} J$
we only need to prove that 
\begin{align*}
d \mu(J)(f)(v)=-\iota_{\rho(f)}\bm\kappa(v), \forall v\in T_J\ac,
\end{align*}
where ${ \bm{\kappa}}(u,v)$ is the K\"ahler form of the K\"ahler manifold $\ac$ defined by 
$$
{ \bm{\kappa}}(u,v)=\int_M\tr(Juv)=\la Ju,v\ra.
$$
We further compute with \eqref{scalarvaria} to get
\begin{align*}
d \mu(J)(f)(v)=\int_Mf(D\sch(v))=-\int_M f(JP)^*v=-\la JP(f), v\ra=-\iota_{\rho(f)}{ \bm{\kappa}}(v).
\end{align*}

%%%%%%%%%%%%%%%%%%%%%%%%%%%%%%%%%%%%%%%%%%%%%%%%%%%%%%%%%%%%%%%%%%%%%%%%%%%%%%%%%%%%%%%%%%%%%%%%%%%%%%%%%%
\subsection{Lichnerowicz operator $\Li$}
Lichnerowicz operator is a 4th order elliptic operator defined on K\"ahler manifolds. It was first introduced in 1958 by Lichnerowicz\cite[Chapter V]{MR0124009}. Later in 1985, Calabi(\cite{MR780039}) used the complex version of Lichnerowicz operator, which is called Calabi operators by Gauduchon in \cite[Section 4.5]{gauduchon2010calabi} , to calculate the variation of  Calabi functional. Gauduchon gave a very detailed and comprehensive introduction of Lichnerowicz operator in his book \cite{gauduchon2010calabi}. In this section, we generalise this notation to almost K\"ahler manifolds.

\begin{defn}\label{myl1}
On an almost K\"ahler manifold $(M,\omega, J)$, the generalised Lichnerowicz oeprator $\Li: C^{\infty}(M)\to  C^{\infty}(M)$ is defined by 
\begin{align*}
\Li(f)=P^*P(f), \ \  f\in C^{\infty}(M,\mathbb R).
\end{align*}
\end{defn}
\begin{rem}
By \lemref{pstar}, we know that $\Li=\frac{1}{2}\delta J(\delta J\Lie_{\g_{\omega}f}J)^{\flat}$. This formula was first studied by Vernier \cite{MR4174303}, who computed the principal term of $\Li$ and proved that $\Li$ is a 4th order elliptic operator.
\end{rem}

Since the Riemannian metric on $M$ is $J$ invariant,  we have
\begin{align*}
\la\Li(f_1), f_2\ra =\la P(f_1), P(f_2)\ra= \la JP(f_1), JP(f_2)\ra, \quad \forall f_1, f_2\in C^{\infty}(M,\R).
\end{align*}
So $\Li$ has an equavilent expression:
$$
\Li = (JP)^*JP.
$$

Its expression implies that  $\Li$ is a self-adjoint semi-positive operator, and we will see the explicit expression of Lichnerowicz operator in Section \ref{sect}. By \defref{myl1}, we have the following description for the kernel of $\Li$.
\begin{prop}\label{ker11}
Let $(M,\omega, J)$ be an almost K\"ahler manifold, then $\Li(f)=0$ iff $\mathcal L_{\g_{\omega}f}J=0$, i.e. the symplectic gradient   $\g_{\omega}f$ is holomorphic.
\end{prop}

In fact, \defref{myl1} is a natural generalisation of Lichnerowicz operator in K\"ahler case.
\begin{prop}\label{equiv1}
When $(M,\omega,J,g)$ is a K\"ahler manifold, \defref{myl1} coincides with the definition $$\Li=(\D^-d)^*\D^-d$$ in the K\"ahler case, where 
$
\D^-\alpha(X,Y)=\frac{1}{2}[(D_X\alpha)Y-(D_{JX}\alpha)(JY)], \forall \alpha\in\Omega^1(M),
$
is the $J$ anti-invariant part of $\D\alpha$.
\end{prop}
\begin{proof}
According to (\cite{gauduchon2010calabi}, Lemma 1.23.2), it holds
\begin{align}\label{abc123}
\D^-df(X,Y)=-\frac{1}{2}g((J\mathcal L_{\g  f}J)X, Y),  \quad\forall f\in C^{\infty}(M,\mathbb R), \quad\forall X,Y\in\Gamma(TM).
\end{align}
The K\"ahler condition gives
$$
\Lie_{J\g f}J -J\mathcal L_{\g  f}J= 4N(\g f, )=0.
$$
Hence we have 
\begin{align}\label{abc1234}
\D^-df(X,Y)=-\frac{1}{2}g((\mathcal L_{J\g  f}J)X, Y)=-\frac{1}{2}g((\mathcal L_{\g_{\omega}  f}J)X, Y).
\end{align}
 It follows from (\ref{abc1234}) that
\begin{align*}
\la \D^-df_1, \D^-df_2\ra
=&\frac{1}{4}\la  \mathcal L_{\g_{\omega} f_1}J, \mathcal L_{\g_{\omega} f_2}J \ra=\la P(f_1), P(f_2)\ra.
\end{align*}
Taking adjoint, we obtain that
$$\la (\D^-d)^*\D^-df_1, f_2\ra=\la\Li(f_1), f_2\ra,$$ i.e. $\Li=(\D^-d)^*\D^-d$.
\end{proof}

%%%%%%%%%%%%%%%%%%%%%%%%%%%%%%%%%%%%%%%%%%%%%%%%%%%%%%%%%%%%%%%%%%%%%%%%%%%%%%%%%%%%%%%%%%%%%%%%%%%%%%%%%%%%%%%%%%%%

\subsection{Operator $\Lie_\K$}
Another important operator related to Lichnerowicz operator is the Lie derivative along $\K$.
\begin{defn}\label{lkop}
For any tensor field $T$  on $M$, we define the operator $\mathcal L_{\K}$ of  Lie derivative along $\K$, i.e. $\mathcal L_{\K}:T\to \mathcal L_{\K}T$.
\end{defn}
 The most important case is $T\in T_J\ac$ and $T$ is a function. When $T=f$ is a function, we have
$$
\mathcal L_{\K}f=(d^c\sch, df)=\{\sch, f\}.
$$

To further describe $\Lie_\K$ acting on functions, we consider the twisted Lichnerowicz operator $(JP)^*P$.  

In fact, we see that 
$(JP)^*P$ is an anti self-adjoint operator on functions. The proof is a direct computation.
It follows from (\ref{def:jpstar}) that
$$(JP)^*P(f)=-P^*JP(f)=-[(JP)^*P]^\ast(f).
$$

By definition $(JP)^*P$ seems to be a 4th order operator. However, it is half of the operator $\Lie_{\K}$. To prove this fact we need a lemma of  Garc\'ia-Prada and Salamon.
\begin{lem}[{\cite[ Remark 2.10]{MR4234988}}]\label{salamon}
For a closed connected symplectic $2m$-manifold $(M,\omega)$ , an almost complex structure $J\in\ac$, and two
Hamiltonian momentum functions $f, g: M \to\mathbb R$ we have
\begin{align}\label{sala0}
\la P(f_1), JP(f_2)\ra= \frac{1}{2}\la\sch,\{f_1, f_2\}\ra.
\end{align}
\end{lem}

\begin{prop}\label{lkprop}
Let $(M,\omega,J)$ be an almost K\"ahler manifold, then for any $f$ we have
\begin{align}
\mathcal L_{\K}(f)=2(JP)^*P(f)=-2P^* JP(f)=-\mathcal L_{\K}^\ast(f).
\end{align}
In particular, $\Lie_\K$ is  anti-self-adjoint.
\end{prop}
\begin{proof}
First we see
\begin{align*}
\la\sch,\{f_1, f_2\}\ra= \la \sch, (d^cf_1, df_2)\ra 
= \la \sch d^cf_1, df_2\ra 
= \la \delta (\sch d^cf_1),f_2 \ra .
\end{align*}
Using $\delta d^cf_1=0$ from \lemref{lem36}, we get
\begin{align}
\delta (\sch d^cf_1)=-(d\sch, d^cf_1)+\sch\delta d^cf_1=(d^c\sch, d f_1)=\mathcal L_{\K}f_1.
\end{align}
So we have
\begin{align*}
\la\sch,\{f_1, f_2\}\ra
= \la\mathcal L_{\K}f_1,f_2 \ra .
\end{align*}
It then follows from (\ref{sala0}) that $\mathcal L_{\K}= 2(JP)^*P= -2P^*JP$.
\end{proof}

\begin{rem}\label{lkequiv}
When $(M,\omega, J)$ is K\"ahler, Gauduchon(\cite[Lemma 1.23.5]{gauduchon2010calabi}) proved that
$$
2\delta\delta\D^-d^cf=-\Lie_{\K}f,
$$
where the $\delta$ operator is the formal adjoint of $\D$ as we defined in (\ref{delta}),  and after taking two successive $\delta $ operation, the (0,2)-tensor $\D^-d^cf$ becomes a function.

In fact, \propref{lkprop} is a generalisation of Gauduchon's result on almost K\"ahler manifolds. 
By definition, it holds
\begin{align*}
\la\delta\delta\D^-d^cf_1, f_2\ra=\la \delta\D^-d^cf_1, df_2\ra
=\la  \D^-d^cf_1, \D df_2\ra.
\end{align*}
Due to the fact $ \D df_2= \D^- df_2+ \D^+ df_2$ and $\la  \D^-d^cf_1, \D^+ df_2\ra=0$, we have
\begin{align*}
\la\delta\delta\D^-d^cf_1, f_2\ra=\la  \D^-d^cf_1, \D^-df_2\ra=\la  J\D^-df_1, \D^-df_2\ra.
\end{align*}

It follows from (\ref{abc1234}) that $$\la  J\D^-df_1, \D^-df_2\ra=\la JP(f_1),  P(f_2)\ra.$$ 
So, the anti self-adjointness of $(JP)^*P$ leads to
\begin{align*}
\la\delta\delta\D^-d^cf_1, f_2\ra
=\la P^*JP(f_1), f_2 \ra=-\la (JP)^*P(f_1), f_2\ra.
\end{align*}
\end{rem}

Now we study  the action of $\Lie_\K$ on $T_J\ac$.

\begin{lem}\label{lemma:anti-lk}
When acting on $T_J\ac$, $\Lie_\K$ is an anti-self-adjoint operator.
\end{lem}
\begin{proof}
From (\ref{tangent}), we know that any $u\in T_J\ac$ is symmetric. The metric on $T_J\ac$ could also be written as (see also \cite[(9.2.10)]{gauduchon2010calabi})
$$
\la u, v\ra=\int_M\tr(uv)\vol,
$$
where $uv=u\circ v$ denote the composition of $u,v\in \en(TM)$ and $\tr$ is the trace operation. Since trace operation commutes with Lie derivative(for example, see \cite[Exercise 2.5.10]{MR3469435}), 
we have
\begin{align*}
\la \mathcal L_{\K}u,v\ra=&\int_M\tr( (\mathcal L_{\K}u)v)=\int_M\tr( \mathcal L_{\K}(uv)-u \mathcal L_{\K}v)
=\int_M\mathcal L_{\K}\tr (uv)-\tr(u \mathcal L_{\K}v).
\end{align*}
But \lemref{lem36} implies $\delta d^c\sch=0$,  we have
\begin{align*}
\int_M\mathcal L_{\K}\tr(uv)=\la d^c\sch, d(\tr(uv))\ra=\int_M-\delta(d^c\sch\tr (uv))=0.
\end{align*}

We obtain
\begin{align*}
\la \mathcal L_{\K}u,v\ra=\int_M-\tr(u \mathcal L_{\K}v)=-\la u,\mathcal L_{\K}v\ra.
\end{align*}
Thus $ \mathcal L_{\K}$ is anti self-adjoint.  

\end{proof}

If $J$ is EAK, the operator $\Lie_\K$ have the following commutative relation.
\begin{lem}\label{exal}
If $J$ is EAK, then $\mathcal L_{\K}$ commutes with $P,P^*,JP, (JP)^*, \Li$.
\end{lem}
\begin{proof}
Since $J$ is EAK, we have
$\mathcal L_{\K}J=0$, which implies
\begin{align}\label{commute0}
\mathcal L_{\K}(P(f))= \frac{1}{2} \mathcal L_{\K}\mathcal L_{\g_{\omega}f }J=\frac{1}{2} (\mathcal L_{\K}\mathcal L_{\g_{\omega}f }-\mathcal L_{\g_{\omega}f }\mathcal L_{\K})J=\frac{1}{2} \mathcal L_{[\K,\g_{\omega}f ]}J.
\end{align}
The Poisson bracket satisfies 
\begin{align}\label{commute1}
[\K,\g_{\omega}f ]=[\g_{\omega}s^{\nabla}_J,\g_{\omega}f ]=\g_{\omega}\{s^{\nabla}_J,f \}=\g_{\omega}\Lie_{\K}f.
\end{align}
Combining (\ref{commute0}) and (\ref{commute1}) together, we arrive at
\begin{align}\label{commute5}
\mathcal L_{\K}P(f)=P\mathcal L_{\K}(f).
\end{align}

It follows from (\ref{commute5}) and \lemref{lemma:anti-lk} that
\begin{align*}
\la f, \Lie_{\K} P^*v\ra= -\la  P\Lie_{\K} f, v\ra=-\la  \Lie_{\K} Pf, v\ra
=
\la  Pf,  \Lie_{\K}v\ra
=
\la f,   P^*\Lie_{\K}v\ra,
\end{align*}
for any $v\in\en(TM), f\in C^{\infty}(M,\R)$, i.e. 
\begin{align}
\label{commute6}
\mathcal L_{\K}P^*=P^*\mathcal L_{\K}.
\end{align}

Since
$$
\mathcal L_{\K}(JP(f))=(\mathcal L_{\K}J)P(f)+  J\mathcal L_{\K}P(f),
$$
making use of the EAK condition $\mathcal L_{\K}J=0$ and (\ref{commute5}), we get
\begin{align}\label{commute7}
\mathcal L_{\K}JP(f)=JP(\mathcal L_{\K}f).
\end{align}

Applying \lemref{lemma:anti-lk} and \eqref{commute7}, we obtain that 
\begin{align*}
\la f, \Lie_{\K} (JP)^*v\ra= -\la J P\Lie_{\K} f, v\ra=-\la  \Lie_{\K}J Pf, v\ra
=
\la  JPf,  \Lie_{\K}v\ra
=
\la f,   (JP)^*\Lie_{\K}v\ra,
\end{align*}
i.e.
\begin{align}
\label{commute8}
\mathcal L_{\K}(JP)^*=(JP)^*\mathcal L_{\K}.
\end{align}

Since $\Li=PP^* $, it follows from (\ref{commute5}) and (\ref{commute6}) that $\mathcal L_{\K}$ commutes with $ \Li$. This completes the proof.
\end{proof}

\begin{lem}\label{jlk}
 When $J$ is EAK, the operator $J\mathcal L_{\K}$ is self-adjoint on $T_J\ac$ and semi-positive on $\IM P$ and $\IM JP$.
\end{lem}
\begin{proof}
According to \lemref{lemma:anti-lk} and the EAK condition $\mathcal L_{\K}J=0$, we have
\begin{align*}
\la J\mathcal L_{\K}u,v\ra=-\la \mathcal L_{\K}u,Jv\ra=\la u,\mathcal L_{\K}(Jv)\ra=\la u,J\mathcal L_{\K}v\ra,
\end{align*}
i.e. $ J\mathcal L_{\K}$ is self-adjoint. Choosing any  $v=P(\phi)\in \IM P$, Proposition \ref{lkprop} and the commutative relation in \lemref{exal} gives
$$
\la J\mathcal L_{\K}v,v\ra= \la J\mathcal L_{\K}(P(\phi) ,P(\phi) \ra 
= \la J P \mathcal L_{\K}\phi   ,P(\phi) \ra=\frac{1}{2}\la\mathcal L_{\K}\phi  ,\mathcal L_{\K}\phi \ra
\geq 0.
$$
Taking $v=JP(\phi)$, we conclude that
$$
\la J\mathcal L_{\K}v,v\ra= \la J\mathcal L_{\K}(JP(\phi)) ,JP(\phi) \ra 
= -\la  P \mathcal L_{\K}\phi   ,JP(\phi) \ra=\frac{1}{2}\la\mathcal L_{\K}\phi  ,\mathcal L_{\K}\phi \ra\geq 0.
$$
\end{proof}

\subsection{Calabi operators $\Li^{\pm}$}
Considering the $(0, 1)$ and $(1, 0)$ part of $P$,
\begin{align}
P^{+}=\frac{1}{2}(P- iJP),\ \ 
\
 P^{-}=\frac{1}{2}(P+ iJP),
\end{align}
we can define Calabi operators on almost K\"ahler manifolds.
\begin{defn}\label{cala12}
On an almost K\"ahler manifold $(M,\omega,J)$, the Calabi operators are defined by
\begin{align*}
\Li^+(f)=2(P^+)^*P^+ f,\ \ \  \Li^-(f)=2(P^-)^*P^- f.
\end{align*}

If we extend $P$ and $P^*$ to the space of complex function on $M$, and consider the Hermitian $L^2$ inner products on $C^{\infty}(M,\mathbb C)$ and $T_J\ac\otimes\mathbb C$, then both $\Li^{\pm}$ are all self-adjoint semi-positive operators.
\end{defn}
\begin{prop}\label{cala13}
By the definition of\  $\Li$ and $\mathcal L_{\K}$, we obtain equivalent expressions of Calabi operators
\begin{align}
\Li^+= \Li +\frac{i}{2}\mathcal L_{\K},\quad  \Li^-= \Li-\frac{i}{2}\mathcal L_{\K}.
\end{align}
\end{prop}
\begin{proof}
Direct computation shows
\begin{align*}
\Li^+(f)=&2(P^+)^*P^+(f)\\
=&\frac{1}{2}(P^*+ i(JP)^*)(P- iJP)\\
=&\frac{1}{2}[P^*P+ (JP)^*JP +i(JP)^*P- iP^*JP]\\
=&\Li +\frac{i}{2}\mathcal L_{\K}.
\end{align*}
Similarly, we can obtain the expression of  $\Li^-$.
\end{proof}

\begin{rem}
In the K\"ahler case, the most commonly used version of Lichnerowicz operator(\cite[Equation(1.2)]{MR780039}, \cite[Corollary 1]{MR1274118}) is 
\begin{align}
\psi\to (\overline{\partial}\partial^{\sharp})^*(\overline{\partial}\partial^{\sharp})\psi,
\end{align}
where $\partial^{\sharp}\psi=(\overline\partial\psi)^{\sharp}$. In fact, $ (\overline{\partial}\partial^{\sharp})^*(\overline{\partial}\partial^{\sharp})$ is half of $\Li^+$. Since $\D^{0,1}(\bar\partial\psi)\in \Omega^{0,1}(M)\otimes \Omega^{0,1}(M)$, we have
$$
\D^{0,1}(\bar\partial\psi)=\D^{-}\bar\partial\psi.
$$
K\"ahler condition implies $\D g=\D J=0$. Thus $\D^{1,0}$ commutes with flatten operator and
\begin{align*}
\la \bar\partial(\partial^{\sharp}\psi_1), \bar\partial(\partial^{\sharp}\psi_2)\ra
=&
\la \D^{0,1}(\partial^{\sharp}\psi_1),\D^{0,1}(\partial^{\sharp}\psi_2)\ra\\
=&
\la \D^{0,1}(\bar\partial\psi_1),\D^{0,1}(\bar\partial\psi_2)\ra \\
=&
\la \D^{-}(\bar\partial\psi_1),\D^{-}(\bar\partial\psi_2)\ra.
\end{align*}
By the relation
$$
\bar\partial = \frac{1}{2}(d-id^c),
$$
and \propref{equiv1}, \remref{lkequiv},  we have
\begin{align*}
\la \bar\partial(\partial^{\sharp}\psi_1), \bar\partial(\partial^{\sharp}\psi_2)\ra
=&
\frac{1}{4}\la \D^{-}d\psi_1-i\D^{-}d^c\psi_1,\D^{-}d\psi_2-i\D^{-}d^c\psi_2\ra
\\
=&\frac{1}{2}(
\la \D^{-}d\psi_1 ,\D^{-}d\psi_2 \ra
-i\la  \D^{-}d^c\psi_1,\D^{-}d\psi_2 \ra)
\\
=&\frac{1}{2}(
\la\Li\psi_1 , \psi_2 \ra
-i\la  \frac{\Lie_{\K}}{2}\psi_1, \psi_2 \ra),
\end{align*}
i.e.
$$
(\overline{\partial}\partial^{\sharp})^*(\overline{\partial}\partial^{\sharp})=\frac{1}{2}(\Li-\frac{i}{2}\Lie_{\K}).
$$
\end{rem}

\begin{prop}\label{kercala}
For any $F=\varphi+i\phi\in C^{\infty}(M,\mathbb C)$, we have $F\in \ker\Li^+$ if and only if
$$
\Li \varphi-\frac{1}{2}\mathcal L_{\K}\phi=0,\ \  \Li \phi+\frac{1}{2}\mathcal L_{\K}\varphi=0.
$$
And  $F\in \ker\Li^-$ if and only if
$$
\Li \varphi+\frac{1}{2}\mathcal L_{\K}\phi=0,\ \  \Li \phi-\frac{1}{2}\mathcal L_{\K}\varphi=0.
$$
In particular, if $F=\varphi\in C^{\infty}(M,\mathbb R)$, then $F\in\ker \Li^{+}$ iff \ $\Li(\varphi)= \mathcal L_{\K}(\varphi)=0$ iff $P(\varphi)=0$. The same conclusion holds for $\varphi\in\ker \Li^{-}$.
\end{prop}
\begin{proof}
For $F=\varphi+\sqrt{-1}\phi\in C^{\infty}(M,\mathbb C)$, we calculate
\begin{align*}
\Li^+(F) = &\Li (\varphi+i\phi)+\frac{i}{2}\mathcal L_{\K}(\varphi+i\phi)\\
=&\Li(\varphi)-\frac{1}{2}\Lie_{\K}(\phi)+i(\Li(\phi)+\frac{1}{2}\mathcal L_{\K}(\varphi)).
\end{align*}
So $\Li^+(F)=0$ if and only if 
$$
\Li \varphi-\frac{1}{2}\mathcal L_{\K}\phi=0,\ \  \Li \phi+\frac{1}{2}\mathcal L_{\K}\varphi=0.
$$
If $F=\varphi$, i.e. $\phi=0$, then $\Li^+(F)=0$ if and only if 
$$
\Li(\varphi)=\mathcal L_{\K}(\varphi)=0.
$$ Here
$\Li(\varphi)=0$ implies $P(\varphi)=0$, which also implies $\Lie_\K(\varphi)=0$ since $\Lie_\K(\varphi)=2(JP)^*P(\varphi)$.
\end{proof}

\begin{rem}
When the background manifold is K\"ahler, $F=\varphi+i\phi\in\ker \Li^+$ implies 
$$
\g \varphi+J\g \phi
$$
is holomorphic (see \cite[Section 2.5]{gauduchon2010calabi}). In the K\"ahler case, if we denote $\mathfrak h_{\rm{red}}$ the set of real holomorphic vector fields whose zero set is non-empty, then any real vector field $X\in \mathfrak h_{\rm{red}}$ if and if and only if there exists $F=\varphi+i\phi\in\ker \Li^+$ such that (\cite[section 95]{lichnerowicz1958geometrie})
$$
X=\g \varphi+J\g \phi.
$$
\end{rem}

When $J$ is EAK, the Calabi operators have the following commutative relation.
\begin{lem}\label{calabicom}
If $J$ is EAK, then $\Li^+,\Li^-$ commute. The composition $\Li^+\Li^-$ is self-adjoint and semi-positive and we have
$$
\ker (\Li^+\Li^-)=\ker \Li^++\ker\Li^-.
$$
\end{lem}
\begin{proof}
Since $J$ is EAK, \lemref{exal} implies that $\Lie_\K\Li=\Li\Lie_\K$. It follows from \propref{cala13} that $\Li^+\Li^-=\Li^-\Li^+$.
Since $\Li^+,\Li^-$ are all semi-positive and self-adjoint, by commutativity, $ \Li^+\Li^-$ is also semi-positive and self-adjoint. Applying the $L^2$-splitting theorem, we obtain that
$$
C^{\infty}(M,\C)=\ker \Li^{\pm}\oplus\IM\Li^{\pm},
$$
and $\Li^{\pm}:\IM\Li^{\pm}\to \IM\Li^{\pm}$ is isomorhism. So we  have 
$$
\ker (\Li^+\Li^-)=\ker\Li^-+\IM\Li^-\cap\ker\Li^+\subset \ker\Li^-+ \ker\Li^+.
$$
But $\ker\Li^-+ \ker\Li^+\subset \ker (\Li^+\Li^-)$ is obvious, due to the commutativity of $\Li^+,\Li^-$.
\end{proof}

%%%%%%%%%%%%%%%%%%%%%%%%%%%%%%%%%%%%%%%%%%%%%%%%%%%%%%%%%%%%%%%%%%%%%%%%%%%%%%%%%%%%%%%%%%%%%%%%%%%
\subsection{Decomposition of $C^{\infty}(M,\R)$ and $T_J\ac$}
\begin{lem}\label{decomp0}
The functions space $C^{\infty}(M,\R)$ has the following orthogonal decompostion
$$
C^{\infty}(M,\R) = \ker P\oplus \IM P^* = \ker JP\oplus \IM (JP)^*
$$
and $\Li:\IM P^*\to \IM P^*$ is an isomorphism.
\end{lem}
\begin{proof}
Since the Lichnerowicz operator $\Li$ is an elliptic (see \cite[equation (11)]{MR4174303}) self-adjoint operator, the splitting theorem of elliptic operator tells us
$$
C^{\infty}(M,\R)  = \IM \Li\oplus \ker \Li,
$$
and $\Li=P^*P:\IM \Li\to \IM \Li$ is an isomorphisom. Since $\ker \Li = \ker P,\ \IM\Li = (\ker P^*P)^{\perp}=(\ker P )^{\perp}=\IM P^*$, the decomposition becomes
$$
C^{\infty}(M,\R) = \ker P\oplus \IM P^*.
$$
The isomorphism $\Li=P^*P:\IM \Li\to \IM \Li $ can be decomposed as 
$$
\IM \Li=\IM P^*\xrightarrow[\cong]{P} \IM P \xrightarrow[\cong]{P^*}\IM P^*=\IM \Li.
$$
Similarly, if we consider $\Li=(JP)^*JP$, we can obtain decompostion with repect to $JP$ and $(JP)^*$.
\end{proof}

\begin{lem}\label{tangentdecm}
The tangent space  $T_J\ac$ has the following decompostion:
\begin{align}
T_J\ac =& \IM P\oplus\ker P^*=\IM JP\oplus\ker (JP)^*,
\end{align}
and
\begin{align}\label{decomp1}
T_J\ac =& T_J\mathcal D \oplus (\ker P^*\cap \ker (JP)^*).
\end{align}
The map $PP^*:\IM P\to\IM P$ is an isomorphism.
\end{lem}
\begin{proof}
Since
$
(\IM P)^{\perp}=\ker P^*
$
, we have
$$
T_J\ac = \IM P\oplus (\IM P)^{\perp}=\IM P\oplus \ker P^*.
$$
By \lemref{decomp0}, we know that $PP^*:\IM P\to\IM P$ is an isomorphism, and
$$
 \IM P \xrightarrow[\cong]{P^*}\IM P^*\xrightarrow[\cong]{P} \IM P=\IM PP^*.
$$
Similarly we can obtain $T_J\ac  =\IM JP\oplus\ker (JP)^*$ and that $JP(JP)^*:\IM JP\to\IM JP$ is an isomorphism. Therefore, (\ref{decomp1}) follows from the fact $ T_J\mathcal D=\IM P+\IM JP$.

\end{proof}

According to the decomposition of $T_J\ac$ in \lemref{tangentdecm}, we discuss the variation of Hermitian Calabi functional in different directions.
\begin{lem}
If we take variation in  different directions in \corref{1stvaria}, we have
\begin{enumerate}
\item If $v\in \ker (JP)^*$, $D\mathcal C(v)=0, \quad Ds^{\nabla}(v)=0$.
\item If $v\in\IM JP$, $v=JP(f)$,\quad $D\mathcal C(v)=-2\la f, \Li(s^{\nabla})\ra, \quad Ds^{\nabla}(v)=-\Li(f)$.
\item If $v\in\IM P$, $v=P(f) $,\quad $D\mathcal C(v)=0, \quad Ds^{\nabla}(v)=-\frac{1}{2}\Lie_\K(f)$.
\end{enumerate}
\end{lem}
\begin{proof}
The above are all direct consequences of \lemref{tangentdecm},
\begin{enumerate}
\item For any $v\in\ker (JP)^*$, it follows immediately from (\ref{scalarvaria})  
$$
D\mathcal C(v)=-2\la (JP)^* v, s^{\nabla}\ra=0.
$$
\item If we choose $v=JP(f)\in\IM JP$ for some $f\in C^{\infty}(M,\R)$, by (\ref{scalarvaria}) we have 
\begin{align}
Ds^{\nabla}(JP(f))=-(JP)^*JP f=-\Li(f),
\end{align}
 and by (\ref{ca1}), 
\begin{align}\label{bianfen1}
D\mathcal C(JP(f))
=-\la JP(f), J\mathcal L_{\K}J \ra
=-2\la JP(f), JP(\sch) \ra
=-2\la f, \Li(s^{\nabla})\ra.
\end{align}
\item If we choose $v=P(f)\in\IM(P) $ for some $f\in C^{\infty}(M,\R)$, then
\begin{align}\label{pjvari}
Ds^{\nabla}(P(f))=-(JP)^*P f=-\frac{1}{2}\Lie_{\K}f.
\end{align}
Since 
$$
\Lie_{\K}\sch = (d^c \sch, d\sch)=(Jd\sch, d\sch)=-(d\sch, Jd\sch)=0,
$$ we get
\begin{align}\label{lj}
D\mathcal C(P(f))=-\la \Lie_{\K}f, \sch\ra
=\la f, \Lie_{\K}(s^{\nabla})\ra=0.
\end{align}

\end{enumerate}
\end{proof}

\section{Hessian of the Hermitian Calabi functional}
In this section, we study the second variation of Hermitian Calabi functional.
\subsection{Proof of \eqref{prop:second eak} and \eqref{prop:second chsc} in \thmref{prop:second main}}

\begin{thm}\label{prop:second}
For any $J\in\ac$ and $u,v\in T_J\ac$, we choose $J(t_1,t_2)$ satisfying $J(0,0)=J, J'_{t_1}(0,0)=u,J'_{t_2}(0,0)=v$.
Then 
\begin{align*}
\Hess \mathcal C(u,v)=- \la \frac{\partial^2 J}{\partial t_1\partial t_2}\vert_{(0,0)}, J\Lie_\K J\ra+2\la  (JP)^*u,(JP)^*v\ra-\la u, v\mathcal L_{\K }J+J\mathcal L_{\K }v\ra.
\end{align*}
In particular,
\begin{enumerate}
\item Assuming that $J$ is EAK, then 
\begin{align}\label{hessc0}
\Hess \mathcal C(u,v)=2\la  (JP)^*u,(JP)^*v\ra-\la J\mathcal L_{\K }u,v\ra.
\end{align}
\item
Assuming that $J$ has constant Hermitian scalar curvature, then
\begin{align} \label{hesschsc}
\Hess \mathcal C(u,v)=2\la  (JP)^*u,(JP)^*v\ra.
\end{align}
$\Hess\mathcal C$ is semi-positive on $T_J\ac$ and $\Hess\mathcal C(v,v)=0$ iff $v\in\ker (JP)^*$.
\end{enumerate}
\end{thm}
\begin{proof}
Taking derivative on (\ref{ca1}) leads to
\begin{align*}
\Hess \mathcal C(u,v)=&-\frac{\partial}{\partial t_2}|_{(0,0)}\la J_{t_1}'(t_1,t_2),J(t_1,t_2)\mathcal L_{\K(t_1,t_2)}J(t_1,t_2)\ra\\
=&
- \la \frac{\partial^2 J}{\partial t_1\partial t_2}|_{(0,0)}, J\Lie_\K J\ra-\la  u,v\Lie_{\K}J+J\Lie_{D\K(v)}J+J\Lie_{\K}v\ra
.
\end{align*}

It follows from \lemref{1stvaria} and the identity: $\iota_{\K}\omega=-d\sch$ that
\begin{align*}
\iota_{{D\K(v)}}\omega=-d D\sch(v) =d(JP)^*v.
\end{align*}
So we have
\begin{align*}
D\K(v) = -\g_{\omega}(JP)^*v,
\end{align*}
which gives
$$J\Lie_{D\K(v)}J=-2JP(JP)^*v.$$
Then it follows
\begin{align*}
\Hess \mathcal C(u,v)=- \la \frac{\partial^2 J}{\partial t_1\partial t_2}\vert_{(0,0)}, J\Lie_\K J\ra+2\la  (JP)^*u,(JP)^*v\ra-\la u, v\mathcal L_{\K }J+J\mathcal L_{\K }v\ra.
\end{align*}
If $J$ is EAK, then $\Lie_K J=0$, we obtain (\ref{hessc0}); if $J$ has constant Hermitian scalar curvature, then $\K=0$, we obtain (\ref{hesschsc}). 

The  average scalar curvature 
\begin{align}
\underline \sch=\left(\int_M\sch(J)\vol\right)\big /\left(\int_M \vol\right)
\end{align}
dose not depend on  $J\in\ac$(see \cite[(9.5.15)]{gauduchon2010calabi}). Thus any $J\in\ac$ with constant  Hermitian scalar curvature take $\underline \sch$ as its   Hermitian scalar curvature. Cauchy inequality implies
\begin{align}
(\underline \sch)^2(\int_M \vol)^2 = (\int_M\sch(J)\vol)^2\leq (\int_M[\sch(J)]^2\vol)(\int_M \vol)= \mathcal C(J)(\int_M \vol)
\end{align}
i.e.
\begin{align*}
\mathcal C(J)\geq \int_M (\underline \sch)^2\vol.
\end{align*}
Thus any $J\in\ac$ with constant  Hermitian scalar curvature is the local minimum point of $\mathcal C$.
\end{proof}

\thmref{prop:second} shows that the Calabi functional is convex at $J$ if $J$ has constant Hermitian scalar curvature. In fact, we can also obtain convexity result at EAK point if we restrict Hermitian Calabi functional to complexified orbit.

Tangent vectors of complexified orbits are all characterised by smooth functions on $M$, we will restrict Hermitian Calabi functional in complexified orbits as following.

\begin{cor}\label{var21}
Suppose that $J$ is EAK. Let $u_1=P(f_1), v_1=P(f_2)$ in \thmref{prop:second}. Then
$$
\Hess\mathcal C(P(f_1), P(f_2))=0.
$$ 
Let $u=P(f_1)$ and $v =JP(f_2)$ in \thmref{prop:second},  we have
\begin{align*} 
\Hess\mathcal C(P(f_1), JP(f_2))=0.
\end{align*}
\end{cor}
\begin{proof}
Taking $u_1=P(f_1), v_1=P(f_2)$ in (\ref{hessc0}), we have
\begin{align*}
\Hess \mathcal C(u_1,v_1)=&2\la  (JP)^*P(f_1), (JP)^*P(f_2)\ra-\la J\mathcal L_{\K }(P(f_1)),P(f_2)\ra.
\end{align*}
It follows from the description of $\Lie_\K$ in \propref{lkprop} 
$$
 2\la (JP)^*P(f_1),(JP)^*P(f_2)\ra=\frac{1}{2}\la \mathcal L_{\K } f_1,\mathcal L_{\K }f_2\ra.
$$
By the commutativity $\mathcal L_{\K }P=P\mathcal L_{\K }$ in \lemref{exal}, we get
$$
\la J\mathcal L_{\K }(P(f_1)),P(f_2)\ra=\la JP(\mathcal L_{\K }f_1),P(f_2)\ra
=\frac{1}{2}\la \mathcal L_{\K }(f_1),\mathcal L_{\K }(f_2)\ra.
$$
So we see 
\begin{align*}
\Hess \mathcal C(u_1,v_1)=0.
\end{align*}
Similarly, if we take $u_2=P(f_1), v_2=JP(f_2)$, we then obtain
\begin{align*}
\Hess \mathcal C(u_2,v_2)=&2\la (JP)^*P(f_1),(JP)^*JP(f_2)\ra-\la J\mathcal L_{\K }(P(f_1)),JP(f_2)\ra\\
=&\la \mathcal L_{\K } f_1, \Li  f_2\ra-\la P\mathcal L_{\K } f_1,P(f_2)\ra\\
=& \la \mathcal L_{\K } f_1, \Li  f_2\ra-\la \mathcal L_{\K } f_1, \Li f_2\ra\\
=&0.
\end{align*}

\end{proof}

\begin{cor}\label{hesscp all}
In fact, if $J$ is EAK, then $\IM P$ annihilates the Hessian of Hermitian Calabi functional in the total space $T_J\ac$, i.e.
$$
\Hess\mathcal C(P(f),v)=0, \ \ \forall v\in T_J\ac.
$$
\end{cor}
\begin{proof}
That's because
\begin{align*}
\Hess C(P(f),v)=&2\la  (JP)^*P(f),(JP)^*v\ra-\la J\mathcal L_{\K }P(f),v\ra\\
=&\la \mathcal L_{\K }(f),(JP)^*v\ra-\la  \mathcal L_{\K }(f),(JP)^*v\ra\\
=&0.
\end{align*}
\end{proof}
\begin{cor}\label{hesscp}
Suppose that $J$ is EAK. We let $u_1=JP(f_1), v_1=JP(f_2)$ in \thmref{prop:second}. Then
\begin{align}\label{secondvar}
\Hess\mathcal C(JP(f_1), JP(f_2))
=&2\la\Li^+(f_1), \Li^-(f_2)\ra.
\end{align}
\end{cor}
\begin{proof}
Taking $u=JP(f_1), u=JP(f_2)$ in (\ref{hessc0}), we have
\begin{align}
\begin{split}
\Hess \mathcal C(u,v)=&2\la  (JP)^*JP(f_1),(JP)^*JP(f_2)\ra-\la J\mathcal L_{\K }JP(f_1),JP(f_2)\ra\\
=&2\la \Li(f_1),\Li(f_2)\ra-\la \mathcal L_{\K }JP(f_1),P(f_2)\ra.
\end{split}
\end{align}
Since $J$ is EAK, \lemref{exal} implies $JP\mathcal L_{\K }=\mathcal L_{\K }JP$, we obtain
\begin{align*}
\Hess \mathcal C(u,v)
&=2\la\Li(f_1), \Li(f_2)\ra-\frac{1}{2}\la\mathcal L_{\K}f_1,\mathcal L_{\K}f_2 \ra.
\end{align*}

On the other hand, we use \propref{cala13} to compute
\begin{align*}
&2\la\Li^+(f_1), \Li^-(f_2)\ra\\
=&2\la\Li(f_1)+\frac{i}{2}\mathcal L_{\K}f_1, \Li(f_2)-\frac{i}{2}\mathcal L_{\K}f_2\ra\\
=&2\la \Li(f_1), \Li(f_2)\ra-\frac{1}{2}\la\mathcal L_{\K}f_1, \mathcal L_{\K}f_2\ra+i(\la\Li(f_1),\mathcal L_{\K}f_2\ra+\la \mathcal L_{\K}(f_1), \Li(f_2)\ra).
\end{align*}
Again the commuting relation $\mathcal L_{\K}\Li=\Li\mathcal L_{\K}$ in Lemma \ref{exal} infers that
\begin{align*}
\la\Li(f_1),\mathcal L_{\K}f_2\ra =-\la\mathcal L_{\K}\Li(f_1),f_2\ra=-\la\Li\mathcal L_{\K}(f_1),f_2\ra=
-\la\mathcal L_{\K}(f_1), \Li f_2\ra.
\end{align*}
Thus we obtain
$$
2\la\Li^+(f_1), \Li^-(f_2)\ra=2\la \Li(f_1), \Li(f_2)\ra-\frac{1}{2}\la\mathcal L_{\K}f_1, \mathcal L_{\K}f_2\ra.
$$
\end{proof}

\subsection{Proof of \eqref{calabih} in \thmref{prop:second main} }
By \corref{var21} we have known that $\Hess\mathcal C$  vanishes on $\IM P$.

We only need to show that  $\Hess\mathcal C$ is strictly positive on the subspace $\IM JP$. It is proved in  \corref{hesscp} that
\begin{align}
\Hess\mathcal C(JP(f_1), JP(f_2))
=&2\la\Li^+(f_1), \Li^-(f_2)\ra.
\end{align}
$\Li^+,\Li^-$ are all self-adjoint(see \defref{cala12}), hence
\begin{align*}
\Hess\mathcal C(JP(f_1), JP(f_2))
=&2\la\Li^-\Li^+(f_1), f_2\ra.
\end{align*}
Since $J$ is EAK, the commutativity of  $\Li^+, \Li^-$ in \lemref{calabicom} implies that  $\Li^-\Li^+$ is semi-positive and $\ker\Li^+\Li^-=\ker\Li^++\ker\Li^-$. By \propref{kercala}, any real function $f\in\ker \Li^{\pm}$ iff $JP(f)=0$, i.e. $JP(f)$ is zero. So $\Hess\mathcal C$ is strictly positive on the subspace $\IM JP$.

\hfill\qedsymbol

\subsection{Proof of \corref{cor1}}For any $v\in \IM P\cap \IM JP$, by \corref{var21}, $v\in\IM P=T_J\mathcal O$ implies that $\Hess\mathcal C(v, v)=0$. But, from \corref{hesscp}, $\Hess\mathcal C$ is strictly positive on the subspace $\IM JP=\mathbb J T_J\mathcal O$, this forces $v=0$,  i.e. $\IM P\cap \IM JP=\{0\}$.

\hfill\qedsymbol

\begin{defn}\label{defofH}
We define the operator $H:T_J\ac\to T_J\ac$ by
\begin{align}\label{def:H}
H(v) = 2JP(JP)^*v-J\Lie_{\K}v, \quad v\in  T_J\ac.
\end{align}
\end{defn}
Assuming that $ J$ is EAK, then $H$ is self-adjoint by \lemref{jlk}, and
$$
\Hess\mathcal C(u,v)=\la H(u), v\ra.
$$
By \corref{var21} and \corref{hesscp}, we know that the operator $H$ is semi-positive over $T_J\mathcal D$, 
$$
H(P(f))=0,\ \  H(JP(f))\geq0,
$$
 and $H(JP(f))=0$ if and only if $JP(f)=0$.

In order the prove \corref{invar}, we introduce  the following lemma.
\begin{lem}\label{kerdes}
If $J$ is EAK, for any $v\in T_J\mathcal D$, if 
$$
\Hess\mathcal C(v, v)=0,
$$
we have $v=P(f)$ for some $f\in C^{\infty}(M,\R)$.
\end{lem}
\begin{proof}
Since $v\in T_J\mathcal D, v=P(f)+JP(g)$ for some $f, g\in C^{\infty}(M,\R)$, using \corref{var21} and \corref{hesscp}, we have
$$
\Hess\mathcal C(v, v)=\Hess\mathcal C(JP(g), JP(g))=0.
$$
The proof given above implies $JP(g)=0$. Thus $v=P(f)$.
\end{proof}

\subsection{Proof of \corref{invar}}  
\begin{enumerate}
\item For any path $J_t$ in the orbit of $\Ham$, we know that $J_t'= P_t(f_t)$ for some $f_t\in C^{\infty}(M,\R)$, by (\ref{lj}) we have
$$
\frac{d}{dt}\mathcal C(J_t)=0,
$$
i.e. $\mathcal C(J_t)$ is constant. Since $\Ham$ is path connected(see \cite[Proposition 10.2]{MR3674984}), the Hermitian Calabi functional is invariant under the action of $\Ham$.
\item 
If $J$ is EAK, we have
$$
T_J\mathcal D = \IM P\oplus \IM JP.
$$
Since $\mathcal C$ is $\Ham$ invariant and $\Hess\mathcal C$  is strictly positive along the directions in $\IM JP$, thus every critical $J$ achieves
a local, non-degenerate minimum value of $\mathcal C$ relative to the action of the gauge
group $\Ham$.

On the other hand, if $J$ is a local minimum, then $D\mathcal C(v)=0,\forall v\in T_J\mathcal D$. Thus for any $f\in C^{\infty}(M,\R), D\mathcal C(JP(f))=\la\Li(\sch), f\ra=0.$
 Thus $\Li(\sch)=0$, which implies $\Lie_\K J=0$.

\item
Suppose that $J_0$ is EAK, denote by $\mathcal E_{J_0}$ the connected component of $J_0$ in the subset of extremal almost K\"ahler metrics in $\mathcal D_{J_0}$. Since $\Ham $ is connected, $\mathcal O_{J_0}$ is already connected, we need to prove that $\mathcal E_{J_0}=\mathcal O_{J_0}$.

We first show that $ \mathcal E_{J_0} \subset\mathcal O_{J_0}$, for any $J_1\in\mathcal E_{J_0}$, choose a curve  $J_t:[0,1]\to \mathcal E_{J_0}$ such that $J(0)=J_0, J(1)=J_1$. Since $J_t$ is EAK for all $t\in [0,1]$, we have $\frac{d}{dt} \mathcal C(J_t)=0$, so 
$$
\mathcal C(t)=\mathcal C(J_t):[0,1]\to \mathbb R
$$ 
is a constant function, any order derivative of $\mathcal C(t)$ is 0, so we have
$$
\frac{d^2}{dt^2} \mathcal C(J_t)=\Hess\mathcal C(J'_t, J'_t)=0.
$$
\lemref{kerdes} implies that $J'_t=P_t(f_t) $ for some $f_t\in C^{\infty}(M,\R)$, this implies that $J_t$ lies in the orbit of $\Ham$ action.

On the other hand, any $J\in\mathcal E_{J_0}$ is a local minimum of $\mathcal C$, but $\mathcal C$ is constant on $\mathcal O_{J_0}$, thus $\mathcal E_{J_0}$ is an open subset of $\mathcal O_{J_0}$, but $\mathcal E_{J_0}$ is closed since it is characterised by the Euler-Lagrange equation
$$
\Lie_{\g_{\omega}\sch_J}J=0.
$$
So $\mathcal E_{J_0}=\mathcal O_{J_0}$.\hfill\qedsymbol 
\end{enumerate}

%\begin{rem}
%Another important problem is the uniqueness of the EAK metric in a complexified orbit. Assuming that $J$ is EAK, is the EAK metric unique module a $\Ham$ group action in $\mathcal D$? Equavilently,  whether $\mathcal E $ has only one connected component? In K\"ahler manifold, the extremal K\"ahler metric in a K\"ahler class is unique module a holomorphic transformation(see \cite{MR2078878},\cite{MR2121892}, \cite[Theorem 4.16]{MR3671939} ). Chen-Sun(\cite{MR3224716}) proved the uniqueness of cscK metrics in the closure of complexified orbit in integral compatible almost complex structure space $\mathcal J^{int}$ module a symplectomorhism.
%\end{rem}

\subsection{Hermitian Calabi functional along geodesic in $\ac$}
In this section, we prove \eqref{prop:second geodesic} in \thmref{prop:second main}.

In order to deduce the Levi-Civita connection and geodesic equation in $\ac$, we introduce the space
\begin{align*}
\End(TM,\omega )=\{A\in\End(TM):\omega (AX,Y)+\omega(X,AY)=0, \forall X,Y\in\Gamma(TM)\}.
\end{align*}
$\End(TM,\omega )$ is a trivial bundle over $\ac$, and $T\ac$ is a sub-bundle of $\End(TM,\omega )$ which is characterized by
$$
T_J\ac=\{A\in \End(TM,\omega ): AJ+JA=0\}.
$$

And we  have
\begin{itemize}
\item
Any secition  $a\in\End(TM,\omega )$ determines a vector field $\hat a$, on $\ac$ by 
\begin{align}\label{defcurve}
\hat a(J)=\frac{d}{dt}|_{t=0}\exp(-ta)J\exp(ta)=[J,a].
\end{align}
\item Conversely for any $A\in T_J\ac$, $a=-\frac{1}{2}JA\in \End(TM,\omega )$ satisfies $A=\hat a$.
\end{itemize}
For any $a\in \End(TM,\omega)$, we denote by $a^{\pm}$ the $J$ commutative and $J$ anti-commutative part of $a$, i.e., 
\begin{align}
a^+=\frac{1}{2}(a-JaJ),\quad  a^-=\frac{1}{2}(a+JaJ),
\end{align}
then we have $a^-\in T_J\ac$ and $\hat a = 2Ja^-$.

\begin{lem}[{\cite[(9.2.7)]{gauduchon2010calabi}}]
For any   $a, b\in \en(TM,\omega )$, we have
\begin{align}\label{geodesic01}
[\hat a, \hat b] = \widehat{[a, b]}=[J, [a, b]],
\end{align}
where $[\hat a, \hat b]$ denote the  Lie bracket of vector fields on $\ac$, and $[a, b]=ab-ba$.
\end{lem}

\begin{lem}\label{geodesiclemma}
For any   $a, b, c\in \en(TM,\omega )$, we have
\begin{align}\label{geodesic02}
[\hat a,   b]-[\hat b, a]=[\hat a, \hat b],
\end{align}
and 
\begin{align}\label{geodesic03}
\hat a\la\hat b, \hat c\ra 
=&\la[\hat a, b], \hat c\ra+\la[\hat a,c], \hat b\ra.
\end{align}
\end{lem}
\begin{proof}
Direct computation shows
\begin{align*}
[\hat a,   b]-[\hat b, a]=& \hat a b- b\hat a- \hat ba+a\hat b \\
=&Jab-aJb-bJa+baJ-Jba+bJa+aJb-abJ\\
=&J[a, b] -[a, b]J\\
=&[J, [a, b]],
\end{align*}
applying (\ref{geodesic01}), we obtain (\ref{geodesic02}) . 

By  (\ref{defcurve}), the definition curve  of $\hat a$ is 
$$J_t = \exp(-ta)J\exp(ta).$$ 
According to definition, we have
\begin{align}\label{geodesic04}
\hat a\la\hat b, \hat c\ra=\frac{d}{dt}|_{t=0}\la\hat b, \hat c \ra(J_t)=\frac{d}{dt}|_{t=0}\la\hat b(J_t), \hat c(J_t)\ra.
\end{align}
Direct computation shows 
\begin{align}\label{geodesic05}
\frac{d}{dt}|_{t=0} \hat b(J_t)=\frac{d}{dt}|_{t=0} [J_t,b]=[\hat a, b].
\end{align}
It follows from (\ref{geodesic04}) and (\ref{geodesic05}) that
\begin{align*}
\hat a\la\hat b, \hat c\ra
=&\la[\hat a, b], \hat c)+\la[\hat a,c], \hat b\ra.
\end{align*}

\end{proof}

Now we turn to study the geodesic equation in $\ac$. Denote by $\mathfrak D$ the Levi-Civita connection on $\ac$, the following lemma is  mentioned in \cite[Section 9.2]{gauduchon2010calabi} and for sake of completeness we give a proof here.
\begin{lem}
For $\hat b\in \Gamma(T\ac), \forall A\in T_J\ac$, we have
\begin{align}\label{geodesic07}
\mathfrak D_{A}\hat b = [A, b]^-.
\end{align}
\end{lem}
\begin{proof}
By Koszul formula, we have
\begin{align*}
2\la\mathfrak D_{\hat a}\hat b, \hat c\ra = &\hat a\la\hat b, \hat c\ra+\hat b\la\hat a, \hat c\ra-\hat c\la\hat a, \hat b\ra+\la[\hat a, \hat b], \hat c\ra-\la[\hat a, \hat c], \hat b\ra-\la[\hat b, \hat c], \hat a\ra.
\end{align*}
It follows from (\ref{geodesic02}) and (\ref{geodesic03}) in \lemref{geodesiclemma} that
\begin{align*}
2\la\mathfrak D_{\hat a}\hat b, \hat c\ra = &\la[\hat a, b], \hat c\ra+\la[\hat a,c], \hat b\ra+\la[\hat b, a], \hat c\ra+\la[\hat b,c], \hat a\ra
-\la[\hat c, b], \hat a\ra-\la[\hat c,a], \hat b\ra\\
&+\la[\hat a, \hat b], \hat c\ra-\la[\hat a, \hat c], \hat b\ra-\la[\hat b, \hat c], \hat a\ra\\
=&
\la[\hat a, b], \hat c\ra+\la[\hat b, a], \hat c\ra
+\la[\hat a, \hat b], \hat c\ra\\
=&2\la[\hat a, b], \hat c\ra.
\end{align*}
Thus $\mathfrak D_{\hat a}\hat b$ is just the $J$-anti-commutative part of $[\hat a, b]^-$, i.e., 
$$\mathfrak D_{\hat a}\hat b=[\hat a, b]^-.$$
Since $\hat a, a\in\en(TM,\omega)$ can generate the whole space $T_J\ac$, (\ref{geodesic07}) establishes.
\end{proof}

By the Levi-Civita connection on $\ac$, we can characterize the geodesic in $\ac$. 
\begin{lem}
A curve $J_t:(a, b)\to\ac$ is a geodesic if and only if $(J_t'')^-=0$, i.e., 
$$
J''J=JJ''.
$$
Using the fact $J'J+JJ'=0$, we can get another equivalent conidtion 
$$
J''=J'J'J.
$$
\end{lem}
\begin{proof}
By definition $J_t$ is a geodesic in $\ac$ if and only if $\mathfrak D_{J_t'}J_t'=0$. For any $b\in\End(TM,\omega)$, we have
\begin{align*}
\la \mathfrak D_{J_t'}J_t', \hat b\ra=&J_t'\la J_t', \hat b\ra-\la J_t', \mathfrak D_{J_t'}\hat b\ra
\\
=&\frac{d}{dt}\la J_t', [J_t, b]\ra-\la J_t', [J_t',b]^-\ra\\
=&\la (J_t'')^-, [J_t, b]\ra+\la J_t', [J_t', b]^-\ra-\la J_t', [J_t',b]^-\ra\\
=&\la (J_t'')^-, [J_t, b]\ra\\
=&\la (J_t'')^-, \hat b\ra.
\end{align*}
\end{proof}

\begin{prop}\label{calonggeo}
 If $J(t)$ is geodesic in $\ac$, we have
\begin{align*}
\frac{d^2}{dt^2}\mathcal C(J_t)=
\la H(J'), J'\ra,
\end{align*}
where $H$ is defined in \defref{defofH}.
\end{prop}
\begin{proof}

Let $J(t)$ be an geodesic in $\ac$, it follows from \thmref{prop:second} that
\begin{align}\label{geod1}
\frac{d^2}{dt^2}\mathcal C(J_t)=&
- \la J'', J\Lie_\K J\ra+2\la  (JP)^* J',(JP)^* J'\ra-\la J', J'\mathcal L_{\K }J+J\mathcal L_{\K } J'\ra.
\end{align}
Due to the geodesic condition that $J''$ is $J$-commutative and the fact that $J\Lie_\K J$ is $J$-anti-commutative, we have
\begin{align}\label{geod2}
\la J'', J\Lie_\K J\ra=0.
\end{align}

For any $u\in T_J\ac$, we have
\begin{align}\label{geod3}
\begin{split}
\la u\mathcal L_{\K }J, u\ra=&\int_M\tr\Lie_K(uJu)-\int_M\tr(\Lie_\K uJu)-\int_M\tr(uJ\Lie_\K u)\\
=&-\la \Lie_\K u, Ju\ra-\la u, J\Lie_\K u\ra\\
=&\la J\Lie_\K u, u\ra-\la u, J\Lie_\K u\ra\\
=&0.
\end{split}
\end{align}

Simplifying (\ref{geod1}) by (\ref{geod2}) and (\ref{geod3}) gives
\begin{align*}
\frac{d^2}{dt^2}\mathcal C(J_t)
=&2\la  (JP)^* J',(JP)^* J'\ra-\la J', J\mathcal L_{\K } J'\ra\\
=&2\la H(J'), J'\ra.
\end{align*}
\end{proof}

\begin{rem}
We don't know whether $H$ is positive in the whole space $\ac$, so the convexity  of Hermitain Calabi functional along geodesic in \propref{calonggeo} is not clear. If $J$ is EAK, we know that the operator $H$ is semi-positive over $T_J\mathcal D$.  It is natural to ask if a geodesic $J(t)$ lies in some complexified orbit $\mathcal D$, whether the Hermitian Calabi functional is convex along the geodesic.
\end{rem}

\section{Hermitian Calabi flow}\label{Hermitian Calabi flow}

According to the variation formula of Hermitian Calabi functional (\ref{ca1}), we write down the gradient flow of Hermitian Calabi functional.
\begin{defn}\label{calabiflow}
The Hermitian Calabi flow(HCF), i.e., gradient flow of Hermitian Calabi functional,  is defined by
\begin{align}
\frac{d}{dt}J = \frac{1}{2}J\Lie_{\K}J=J P(\sch_J).
\end{align}
\end{defn}
Since $J P(\sch_J)\in\IM JP\subset T_J\mathcal D$, we see the Hermitian Calabi flow starting from $J$ always lies in the complexified orbit $\mathcal D_J$.

\begin{prop}
Along the Hermitian-Calabi flow, we have
\begin{align}
\frac{d}{dt}\sch_{J_t} =& -\Li(\sch),\\
\frac{d}{dt}\mathcal C(J_t) =& -\la\frac{1}{2}J\Lie_{\K}J,J\Lie_{\K}J\ra=-2\la\Li(\sch), \sch\ra\leq 0.
\end{align}
So the Hermitain Calabi functional is strictly decreasing along the flow unless $J$ is EAK.
\end{prop}
\begin{rem}
In K\"ahler case, when the flow $J(t)$ is integrable, the Hermitian Calabi flow coincides with the classical Calabi flow 
$$
\frac{d}{dt}\varphi = s(\omega_{\varphi})-\hat s,\ \  \omega_{\varphi} = \omega+\sqrt{-1}\partial\bar\partial\varphi
$$
up to a diffeomorphism (c.f. \cite[Lemma 5.1]{MR3224716}). In \cite{MR3969453}, Li-Wang-Zheng proved the convergence theorems of the Calabi flow on extremal K\"ahler surfaces, which partially confirm Donaldson's conjectural picture \cite{MR2103718}
for the Calabi flow in complex dimension 2, see \cite{MR3969453} for more references therein.
\end{rem}

In this section, we
choose a local coordinate system $(x_1,\ldots, x_n)$. For simplicity, we denote $\frac{\partial}{\partial x_i}$ by $\partial_i$ and we denote the covariant derivative $\D_i=\D_{\frac{\partial}{\partial x_i}}$. 

\begin{lem}\label{order LJ}Let $X=X^i\partial_i$ be a vector field. Then
\begin{align*}
\Lie_{X}J=(X^i \partial_iJ^q_p -J^i_p\partial_i X^q  +J_i^q\partial_p X^i )\partial_q\otimes dx^p,\quad (J\Lie_{X}J)_p^q=J_l^q(J\Lie_{X}J)_p^l.
\end{align*}
If $X=\g_{\omega} f$, then
\begin{align*}
(\Lie_{X}J)^q_p= \omega^{ij}\frac{\partial f}{\partial x_j}\partial_iJ^q_p 
-J^i_p \omega^{qj}\frac{\partial^2f}{\partial x_i\partial x_j} 
-J^i_p\partial_i \omega^{qj}\frac{\partial f}{\partial x_j}  
+J_i^q\partial_p  \omega^{ij}\frac{\partial f}{\partial x_j} 
 +J_i^q  \omega^{ij}\frac{\partial^2f}{\partial x_p\partial x_j}
\end{align*} 
where $\omega^{ij}=g^{kj}J_k^i$ and 
\begin{align*}
(J\Lie_{X}J)^q_p=J_l^q \omega^{ij}\frac{\partial f}{\partial x_j}\partial_iJ^l_p 
-J_l^qJ^i_p \omega^{lj}\frac{\partial^2f}{\partial x_i\partial x_j} 
-J_l^qJ^i_p\partial_i \omega^{lj}\frac{\partial f}{\partial x_j}  
-\partial_p  \omega^{qj}\frac{\partial f}{\partial x_j} 
-  \omega^{qj}\frac{\partial^2f}{\partial x_p\partial x_j}.
\end{align*} 
\end{lem}
\begin{proof}
It is a direct computation
\begin{align*}
\Lie_{X}J
=[\omega^{ij}\frac{\partial f}{\partial x_j}\partial_iJ^q_p -J^i_p\partial_i(\omega^{qj}\frac{\partial f}{\partial x_j})  +J_i^q\partial_p ( \omega^{ij}\frac{\partial f}{\partial x_j})]\partial_q\otimes dx_p.
\end{align*}
The local formula of $(J\Lie_{X}J)^q_p$ follows from substituting 
\begin{align*}
J_l^qJ_i^l=-\delta_i^q
\end{align*}
into in the last two terms in
\begin{align*}
(J\Lie_{X}J)^q_p
&=J_l^q \omega^{ij}\frac{\partial f}{\partial x_j}\partial_iJ^l_p 
-J_l^qJ^i_p \omega^{lj}\frac{\partial^2f}{\partial x_i\partial x_j} 
-J_l^qJ^i_p\partial_i \omega^{lj}\frac{\partial f}{\partial x_j}  \\
&+J_l^qJ_i^l\partial_p  \omega^{ij}\frac{\partial f}{\partial x_j} 
 +J_l^qJ_i^l  \omega^{ij}\frac{\partial^2f}{\partial x_p\partial x_j}.
\end{align*} 
\end{proof}

\begin{lem}\label{order}
The Hermitian Calabi flow defined in \defref{calabiflow} is a 4th order flow.
\end{lem}
\begin{proof}
 First of all, we see $\K = \omega^{ij}\frac{\partial\sch}{\partial x_j}\partial_i$ and $\omega^{ij}=g^{kj}J_k^i$ does not depend on $J$.

Secondly, inserting $X=\K$ and $f=\sch$ in \lemref{order LJ},  we have
the expression of $J\Lie_{\K}J$ in local coordinates 
\begin{align*}
(J\Lie_{\K}J)_p^q =-J_l^qJ^i_p \omega^{lj}\frac{\partial^2\sch}{\partial x_i\partial x_j} 
 -   \omega^{qj}\frac{\partial^2\sch}{\partial x_p\partial x_j}  +L_1,
\end{align*}
where $L_1$ denotes the lower order derivative terms of $\sch$. 

In conclusion, since $\sch$ depends on  2nd order derivative of $J$, we have
the highest order derivative of $J$ involved in $(J\Lie_{\K}J)_p^q$ is 4, thus Hermitian Calabi flow is a 4th order flow.
\end{proof}

\subsection{Linearisation operator of $J\Lie_{\K}J$}

We define
\begin{defn}
$\widetilde H:=D(J\Lie_{\K}J),\quad T_J\ac\to T_{J\Lie_{\K}J}T_J\ac.$
\end{defn}

We first compute $D\K$.
\begin{lem}Let $X=\g_{\omega} f$ a Hamiltonian vector field with $f$ depending on $J$. Then 
$$
D X(v) = \g_{\omega} Df(v).
$$
In particular,
$
D\K(v) = -\g_{\omega}(JP)^*v.
$
\end{lem}
\begin{proof}
We use the identity: $\iota_{X}\omega=-d f$ to see
\begin{align*}
\iota_{{DX(v)}}\omega=-d D f(v).
\end{align*}So, the first identity is obtained.
The variation of $\K$ follows from \lemref{1stvaria}.

\end{proof}

Then we have $D(\Lie_{X}J)$.
\begin{lem}Let $X$ as given above. Then
$$
D(\Lie_{X}J)(v)=2P(Df(v))+\Lie_{X}v.
$$
\end{lem}
\begin{proof}
It follows from $
D(\Lie_{X}J)(v)=\Lie_{\g_{\omega} Df(v)}J+\Lie_{X}v
$ and the definition of $P$.
\end{proof}

Now, we compute the linearisation operator $\widetilde H$. 
\begin{lem}\label{hwan}
$
\widetilde H(v) = -2JP(JP)^*v+J\Lie_\K v+v\Lie_\K J.
$
\end{lem}
\begin{proof}
We continue to calculate by taking $X=\K$ in the previous lemmas
\begin{align*}
D(J\Lie_{\K}J)(v)=&v\Lie_{\K}J+J\Lie_{D\K(v)}J+J\Lie_{\K}v.
\end{align*}
Inserting the formula of $D\K$, we obtain
\begin{align*}
D(J\Lie_{\K}J)(v)
=&2vP(\sch)-2JP(JP)^*v+J\Lie_\K v .
\end{align*}
Thus we prove the lemma.
\end{proof}

\begin{lem}
The principal terms of $\widetilde H$ lie in $-2JP(JP)^*v$ and 
\begin{align}\label{kerppstar}
\ker JP(JP)^* =\{P(f)+JP(\phi), \Lie_\K(f)+2\Li(\phi)=0\}.
\end{align}
In particular, if $J$ has constant Hermitian Calabi functional, then
\begin{align*}
\ker JP(JP)^* =\{P(f), \Lie_\K f=0\}.
\end{align*}
\end{lem}
\begin{proof}
Since $J\Lie_\K v$ is a first order derivative of $v$ and $v\Lie_\K J$ does not involve derivative terms of $v$, the principal term of $\widetilde H$ is contained in $-2JP(JP)^*v$.

Now we consider the kernel space
\begin{align*} 
\ker(JP(JP)^*)=\ker (JP)^*=(\IM JP)^{\perp}.
\end{align*}
Since the HCF always lies in the complexified orbit, we only need to consider $\ker (JP)^* $ in $T_J\mathcal D=\IM P+\IM JP$. For any $v = P(f)+JP(\phi)\in\ker (JP)^*$, we have
\begin{align*}
(JP)^*v=(JP)^*P(f)+(JP)^*JP(\phi) = \frac{1}{2}\Lie_\K(f)+\Li(\phi)=0.
\end{align*}
If $J$ has constant Hermitian scalar curvature, then $\IM P\perp\IM JP$, which implies $(\IM JP)^{\perp}=\IM P$. So we have
\begin{align*}
\ker(JP(JP)^* )=(\IM JP)^{\perp}=\IM P,
\end{align*}
Hence any $v\in \ker JP(JP)^* $ takes the form  $v=P(f)$ for some $f\in C^{\infty}(M,\R)$. Thus we have shown that
$$
(JP)^*v= (JP)^*P(f)=\frac{1}{2}\Lie_\K f=0.
$$
\end{proof}

\subsection{Weak parabolicity of Hermitian Calabi flow}

\begin{lem}\label{induce:tang}
Any 1-form $\xi\in T^*M$  induces an element $\Xi\in T_J\ac$ by
$$
\Xi = \xi^{\sharp}\otimes (J\xi)+ (J\xi^{\sharp})\otimes \xi.
$$
In local coordinates, we denote $\Xi = \Xi^d_c\partial_d\otimes dx_c$. Then 
$$
\Xi^d_c = -\xi^{d}\xi_bJ^b_c+ J_b^d\xi^{b} \xi_c,
$$
and 
$$
(\Xi,v) =2(J\xi^{\sharp}, v(\xi^{\sharp})),\quad \forall v\in T_J\ac.
$$
\end{lem}
\begin{proof}
We need to prove that $J\Xi(X)+\Xi J(X)=0$ and $(\Xi(X),Y)=(X,\Xi(Y))$, for any $X,Y\in TM$. Since
\begin{align*}
J\Xi(X)+\Xi J(X)=&J(J\xi(X)\xi^{\sharp}+  \xi(X)J\xi^{\sharp})+J\xi(JX)\xi^{\sharp}+  \xi(JX)J\xi^{\sharp}\\
=&-\xi(JX)J\xi^{\sharp}-\xi(X) \xi^{\sharp}+\xi(X) \xi^{\sharp}+ \xi(JX)J\xi^{\sharp}\\
=&0
\end{align*}
and 
\begin{align*}
(\Xi(X),Y)=&(J\xi(X)\xi^{\sharp}+  \xi(X)J\xi^{\sharp}, Y)=J\xi(X)\xi(Y)+  \xi(X)J\xi(Y)\\
=&(X,\Xi(Y)).
\end{align*}
Thus $\Xi\in T_J\ac$. 

In local coordinates $(x_1,\ldots, x_n)$, we have
\begin{align*}
\Xi^d_c =& dx_d(\Xi(\partial_c))=  dx_d(\xi^{\sharp}J\xi(\partial_c)+ J\xi^{\sharp} \xi(\partial_c))=-\xi^{d}\xi_bJ^b_c+ J_b^d\xi^{b} \xi_c,
\end{align*}
and 
$$
(\xi,v) =g^{ic}g_{jd}(-\xi^{d}\xi_bJ^b_c+ J_b^d\xi^{b} \xi_c)v^j_i=-(Jv\xi,\xi)+(J\xi,v\xi)=2(J\xi,v\xi).
$$
Here we use the fact $g^{ic}g_{jd}v^j_i = v^c_d$, since $v$ is self-adjoint.
\end{proof}

We use $L_i, i=1,2,\ldots$ to denote the terms containing the derivatives with orders strictly less than 2. 

For $v = v^p_l\frac{\partial}{\partial x_p}\otimes dx_l\in T_J\ac$ we first compute $(JP)^*v=\delta J(\delta v)^{\flat}$ in local coordinates.
\begin{lem}\label{jpstarloc}
\begin{align*}
-D\sch_J(v)=(JP)^*v=-g^{ij}g^{kl}\omega_{jp}\D_i  \D_kv_l^p  +L_3.
\end{align*}
\end{lem}
\begin{proof}
We compute that
\begin{align*}
\delta J(\delta v)^{\flat}=&-g^{ij}\D_i( J(\delta v)^{\flat})_j=-g^{ij}\D_i[ J(\delta v)^{\flat}(\partial_j)]
-
g^{ij}( J(\delta v)^{\flat})(\D_i\partial_j).
\end{align*}

The 2rd derivative terms only lie in $-g^{ij}\D_i[ J(\delta v)^{\flat}(\partial_j)]$. We compute
\begin{align*}
-g^{ij}\D_i[ J(\delta v)^{\flat}(\partial_j)]=&-g^{ij}\D_i (g^{kl}\D_kv(\partial_l),J\partial_j)  \\
=&-g^{ij}g^{kl}\D_i (\D_kv_l^p\partial_p,J\partial_j) +L_1\\
=&-g^{ij}g^{kl}\omega_{jp}\D_i  \D_kv_l^p  +L_2.
\end{align*}
So the lemma is proved.

\end{proof}
Then we compute $JP$ in local coordinates in terms of covariant derivatives as well.
\begin{lem}\label{jploc}
\begin{align*}
JP(f)=J\Lie_{\g_{\omega}f}J =[ J^b_c  g^{ad} \D_b\D_a f-g^{ab}J^d_b\D_c\D_a f]\frac{\partial }{\partial x^d}\otimes dx_c+L_5.
\end{align*}
\end{lem}
\begin{proof}
Due to \lemref{order LJ}, we have
\begin{align*}
J\Lie_XJ =& [J_b^d(\D_X(J\partial_c))^b-J_l^dJ^b_c(\D_bX)^l-(\D_cX)^d]\frac{\partial }{\partial x^d}\otimes dx_c\\
=&-[J^b_cJ^d_l\D_bX^l+\D_cX^d]\frac{\partial }{\partial x^d}\otimes dx_c+L_4
\end{align*}
and $\g_{\omega}f=(g^{ab}J^l_b \D_a f)\partial_l$, we get
\begin{align*}
J\Lie_{\g_{\omega}f}J =-[ J^b_cJ^d_lg^{ak}J^l_k\D_b\D_a f+g^{ab}J^d_b\D_c\D_a f]\frac{\partial }{\partial x^d}\otimes dx_c+L_5.
\end{align*}
While $\sum_lJ^d_lJ^l_k=-\delta^d_k$, we thus prove this lemma.

\end{proof}

\begin{prop}\label{symbol}
For a covector $\xi\in T_x^*M$, the principal symbol of $ JP(JP)^*$ is 
$$
\hat\sigma_4(x,\xi)v  = \frac{1}{2}(v,\Xi)\Xi,
$$
where $\Xi\in T_J\ac$ is defined in \lemref{induce:tang}. So we have
$$
(\hat\sigma_4(x,\xi)v,v)=\frac{1}{2}(v,\Xi)^2\geq 0,
$$
and if we choose $v=J\Xi$, then $(\hat\sigma_4(x,\xi)v,v)=0$.
Thus the Hermitian Calabi flow is a 4th order weakly parabolic system.
\end{prop}
\begin{proof}
According to \lemref{order},  $JP(JP)^*$ is  a 4th order operator, we only need to compute the 4th order derivative terms of $JP(JP)^*$.

Combined \lemref{jpstarloc} and \lemref{jploc}, the principal term of $2JP(JP)^*v$ becomes 
\begin{align*}
 -[J^b_c  g^{ad} g^{ij}g^{kl}\omega_{jp}\D_b\D_a \D_i \D_kv_l^p -g^{ab}J^d_bg^{ij}g^{kl}\omega_{jp}\D_c\D_a \D_i  \D_kv_l^p ] \frac{\partial }{\partial x^d}\otimes dx_c,
\end{align*}
i.e., for any non-zero covector $\xi=\xi_i dx_i$, it holds that
$$
(\hat\sigma_{ JP(JP)^*}(x,\xi)v)^d_c = -J^b_c g^{ad}g^{ij}g^{kl}\omega_{jp}  \xi_b\xi_a \xi_i  \xi_kv_l^p
+g^{ab}J^d_bg^{ij}g^{kl}\omega_{jp}\xi_c\xi_a \xi_i \xi_kv_l^p.
$$

Since
$$
J^b_c g^{ad}g^{ij}g^{kl}\omega_{jp}  \xi_b\xi_a \xi_i  \xi_kv_l^p=J^b_c \xi_b\xi^d  (J\xi^{\sharp}, v(\xi^{\sharp}))
$$
and
$$
g^{ab}J^d_bg^{ij}g^{kl}\omega_{jp}\xi_c\xi_a \xi_i \xi_kv_l^p =  J^d_b \xi_c\xi^b  (J\xi^{\sharp},v(\xi^{\sharp})),
$$
we further have
\begin{align*}
(\hat\sigma_{ JP(JP)^*}(x,\xi)v)^d_c =&(J\xi^{\sharp},v\xi^{\sharp})( -J^b_c \xi_b\xi^d 
+J^d_b \xi_c\xi^b ).
\end{align*}

By \lemref{induce:tang}, we thus obtain
\begin{align*}
(\hat\sigma_{ JP(JP)^*}(x,\xi)v)^d_c =& \frac{1}{2}(\Xi,v)\Xi^d_c.
\end{align*}
i.e.
\begin{align*}
\hat\sigma_{ JP(JP)^*}(x,\xi)v
=&\frac{1}{2}(v,\Xi)\Xi,\quad  \Xi = \xi^{\sharp}\otimes J\xi+ J\xi^{\sharp}\otimes \xi.
\end{align*}
\end{proof}

\appendix

\section{Explicit expression of Lichnerowicz operator}\label{sect}
In this section, we give the explicit expression of Lichnerowicz operator:
\begin{align}\label{lich}
\Li(f)=&\frac{1}{2}\Delta^2f-2(\delta\ric^+, df)+2(\rho,dd^c f)+\delta\delta(\D^{+}df-\D^{-}df),
\end{align}
where $\ric^+$ is defined in (\ref{ric+}), $\rho$ is defined in (\ref{rho}), and
$$
\D^{\pm}df(X,Y)=\frac{1}{2}(\D df(X,Y)\pm\D df(JX,JY)).
$$
Our result is in fact a continuous computation of Vernier \cite{MR4174303}, where the expression of $\Li(f)$ was given
by a $\Delta^2$ term plus an error term. We will write down the error term explicitly.

By \lemref{pstar},  Lichnerowicz operator has an equivalent expression:
$$
\Li(f)=P^*P(f)=\frac{1}{2}\delta\Bigl\{J\bigl[\delta(J\mathcal L_{\g_{\omega}f}J)\bigr]^{\flat}\Bigr\}.
$$
We now begin to calculate $\delta (J\mathcal L_{\g_{\omega}f}J)$. Since
$$
(J\mathcal L_{\g_{\omega}f}J)(X)=-\mathcal L_{\g f} J(X)-4N(\g_{\omega}f,X),
$$
defining $\mathcal N_f(X):=N(\g_{\omega}f,X)$, we have
$$
\delta(J\mathcal L_{\g_{\omega}f}J)=-\delta(\mathcal L_{\g f} J)-4\delta \mathcal N_f.
$$
Our whole proof of (\ref{lich}) is divided into two parts
$$
\Li(f)=-\frac{1}{2} \delta\Bigl\{I+4J \circ II\Bigr\}, \quad I=J[\delta(\mathcal L_{\g_{\omega}f}J)]^{\flat},\quad II=[\delta\mathcal N_f]^{\flat}.
$$ In the first part, we deal with $I$, and in the second part, we deal with $II$. 

We first introduce the following lemmas.
\begin{lem}[{\cite[Lemma 3.19]{MR2778451}}]\label{gradvar}
Let $(M,g)$ be a Riemannian manifold,  $\psi_t$ be the flow of the vector field $\xi\in\Gamma(TM)$, we have
\begin{align}
\frac{d}{dt}|_{t=0}\D^{\psi^*_tg}_XY
=\D^2_{X,Y}\xi+R(X,\xi)Y, \forall X,Y\in\Gamma(TM).
\end{align}
\end{lem}

\begin{lem}[{\cite[Lemma 2.2]{MR2747965}}]\label{bas1}
For any real 1-form $\alpha$, 
\begin{align*}
(\delta\D^{+}\alpha-\delta\D^{-}\alpha)(X)=&-\rho^*(JX,\alpha^{\sharp})-\sum_{j=1}^n\D_{Je_j}\alpha((\D_{e_j}J)X),
\end{align*}
where the star Ricci form $\rho^*$ is defined in (\ref{rho*}).
\end{lem}

\begin{lem}\label{bas2}
For any $\xi\in\Gamma(TM)$, 
\begin{align*}
(\delta\D\xi)^{\flat}=\delta\D(\xi^{\flat}).
\end{align*}
\end{lem}

\begin{proof}
For any 1-form $\alpha$ on $M$, we define  $\gamma,\theta\in\Omega(M)$ by
\begin{align}
\gamma(X)=\D_X\alpha(\xi),\ \eta(X)=\alpha(\D_X\xi), \ \  \forall X\in\Gamma(TM),
\end{align}
then we have
\begin{align*}
\di\gamma=-\la\D\alpha,\D\xi\ra+\delta\D\alpha(\xi),\ \di\eta = -\la\D\alpha,\D\xi\ra+\alpha(\delta\D\xi),
\end{align*}
where locally $(\D\alpha,\D\xi)=g^{ij}\D_i\alpha(\D_j\xi)$, since $\int_M\di \gamma=\int_M\di\eta=0$, we have
\begin{align*}
\la\alpha, \delta\D\xi^{\flat}\ra=\la\delta\D\alpha, \xi^{\flat}\ra=\int_M\delta\D\alpha(\xi)=\int_M\alpha(\D^*\D\xi)=(\alpha, (\delta\D\xi)^{\flat}).
\end{align*}
\end{proof}

The first part has been partially calculated in \cite{MR4174303} in a local orthonormal frame $\{e_i, \ldots, e_n\}=\{e_1, e_2, \ldots e_m, Je_1, \ldots,Je_m\}$. We will follow the computation and give a global expression.
\begin{lem}\label{l1}
For any $X\in\Gamma(TM)$, 
\begin{align*}
I(X)
=
-2\delta(\D^{+}df-\D^{-}df)(X)-2\rho^*(JX,\g f)
-\Delta df(X)+2\ric(\g f, X).
\end{align*}
\end{lem}
\begin{proof}
Denote $\psi_t$ the flow of the vector field $\g f$. By \corref{identity}, we have $\psi^*_t(\delta J)=\delta^{g_t}J_t=0$ where $g_t=\psi_t^*g, J_t = \psi_t^*J$. Acting  $\frac{d}{dt}|_{t=0}$ on $\delta^{g_t}J_t=0$ ,  we have
\begin{align}\label{lich10}
\delta (\mathcal L_{\g f}J)=\delta (\frac{d}{dt}|_{t=0}\psi_t^*J)=-\frac{d}{dt}|_{t=0}\delta^{\psi^*_tg}J=\frac{d}{dt}|_{t=0}\sum_{i, j=1}^n(\psi_t^*g)^{ij}\D^t_{e_i}(J)(e_j).
\end{align}
It follows from  (\ref{lich10}) that
\begin{align}
\label{lich11}
\begin{split}
\delta (\mathcal L_{\g f}J)
=&\sum_{i, j=1}^n\frac{d}{dt}|_{t=0}(\psi_t^*g)^{ij}\D_{e_i}(J)e_j+\delta_{ij}\frac{d}{dt}|_{t=0}\D^t_{e_i}(J)(e_j)\\
=&\sum_{i, j=1}^n-2(\Hess f)_{ij}\D_{e_i}(J)e_j+\delta_{ij}\frac{d}{dt}|_{t=0}[\D^t_{e_i}(Je_j)-J(\D^t_{e_i}e_j)].
\end{split}
\end{align}
By \lemref{gradvar}, we have
the expression of $\delta_{ij}\frac{d}{dt}|_{t=0}[\D^t_{e_i}(Je_j)-J(\D^t_{e_i}e_j)]$, which is
\begin{align}
\begin{split}
=&\sum_{i=1}^n[\D^2_{e_i,Je_i}\g f+R(e_i,\g f)Je_i-J\D^2_{e_i,e_i}\g f-JR(e_i,\g f)e_i]\\
=&\sum_{i=1}^n[\D^2_{e_i,Je_i}\g f+R(e_i,\g f)Je_i+J\delta\D\g f-J\ric(\g f)].
\end{split}
\end{align}
Thanks to the form of the local frame $\{e_i\}$, we have
\begin{align}\label{lich12}
\sum_{i=1}^n\D^2_{e_i,Je_i}\g f=\frac{1}{2}\sum_{i=1}^n\D^2_{e_i,Je_i}\g f-\D^2_{Je_i,e_i}\g f=-\frac{1}{2}R(e_i,Je_i)\g f.
\end{align}
The Bianchi identity gives
\begin{align}
\label{lich13}
\begin{split}
\sum_{i=1}^nR(e_i,\g f)Je_i=&\frac{1}{2}\sum_{i=1}^nR(e_i,\g f)Je_i-R(Je_i,\g f)e_i\\
=&\frac{1}{2}\sum_{i=1}^nR(e_i,Je_i)\g f.
\end{split}
\end{align}
Since (\ref{lich12}) and (\ref{lich13}) add up to zero, (\ref{lich11}) becomes
\begin{align}\label{lich14}
\delta (\mathcal L_{\g f}J)
=&-2\D_{e_i}(J)\D_{e_i}\g f+J\delta\D \g f-J\ric(\g f).
\end{align}
Applying \lemref{bas2} to the second term of the right hand side of (\ref{lich14}), we get
\begin{align*}
J(\delta (\mathcal L_{\g f}J))^{\flat}(X)=2\D_{Je_i}df(\D_{e_i}(J)X)-\delta\D df(X)+\ric(\g f, X).
\end{align*}
 Using the Weitzenbock formula on 1-form, we further write
\begin{align*}
J(\delta (\mathcal L_{\g f}J))^{\flat}(X)=2\D_{Je_i}df(\D_{e_i}(J)X)-\Delta df(X)+2\ric(\g f, X).
\end{align*}
Inserting \lemref{bas1}, we have the resulting identity.
\end{proof}

Now let's compute the second part $II$, we first introduce the following lemma.
\begin{lem}\label{aaaa}
Let $(M,g, J,\omega)$ be an almost K\"ahler manifold, let $\{e_1, \ldots , e_n\}$ be an orthonormal frame, then for any vector field $X\in\Gamma(TM)$, we have
\begin{align*}
\sum_{i=1}^n(\D^2_{X,e_i}J)e_i=0.
\end{align*}
\end{lem}
\begin{proof}
Direct computation shows
\begin{align*}
\sum_{i=1}^n(\D^2_{X,e_i}J)e_i=&\sum_{i=1}^n[(\D_X\D_{e_i}J)e_i-(\D_{\D_Xe_i}J)e_i]\\
=&\sum_{i=1}^n\D_X((\D_{e_i}J)e_i)-(\D_{e_i}J)\D_Xe_i-(\D_{\D_Xe_i}J)e_i.
\end{align*}
\lemref{identity} implies $\delta J=-\sum_{i=1}^n(\D_{e_i}J)e_i=0.$ So we have
\begin{align}
\label{lichequa01}
\begin{split}
\sum_{i=1}^n(\D^2_{X,e_i}J)e_i
=&\sum_{i=1}^n-(\D_{e_i}J)\D_Xe_i-(\D_{\D_Xe_i}J)e_i\\
=&\sum_{i, j=1}^n-g(\D_Xe_i, e_j)(\D_{e_i}J)e_j-g( \D_Xe_i, e_j)(\D_{e_j}J)e_i\\
=&\sum_{i, j=1}^n-(g(\D_Xe_i, e_j)+g( e_i, \D_Xe_j))(\D_{e_i}J)e_j.
\end{split}
\end{align}
The last line in (\ref{lichequa01}) vanishes since  $g(e_i,e_j)=\delta_{ij}$, 
$$
g(\D_Xe_i, e_j)+g( e_i, \D_Xe_j)=Xg(e_i,e_j)=0.
$$
\end{proof}

\begin{lem}\label{l2}
For any $X\in\Gamma(TM)$, 
\begin{align*}
II(X)=(\delta \mathcal N_f)^{\flat}(X)
=-\frac{1}{2}\rho^*(X,\g f)-\frac{1}{2}\ric(X,\g_{\omega}f).
\end{align*}
\end{lem}
\begin{proof}
Choosing an auxiliary local orthonormal frame $\{e_1, \ldots, e_n\}$, we have
\begin{align*}
\delta \mathcal N_f=&-\sum_{i=1}^n\D_{e_i}(\mathcal N_f(e_i))+\sum_{i=1}^n\mathcal N_f(\D_{e_i}e_i)\\
=&-\sum_{i=1}^nD_{e_i}N(\g_{\omega}f,e_i)+\sum_{i=1}^nN(\g_{\omega}f,\D_{e_i}e_i).
\end{align*}
Thus 
\begin{align*}
&(\delta \mathcal N_f)^{\flat}(X)\\
=&-g( D_{e_i}N(\g_{\omega}f,e_i), X)+g( N(\g_{\omega}f,\D_{e_i}e_i), X)\\
=&-D_{e_i}g( N(\g_{\omega}f,e_i), X)+g( N(\g_{\omega}f,e_i), D_{e_i}X)+g( N(\g_{\omega}f,\D_{e_i}e_i), X).
\end{align*}
Here, we omit the summary notation $\sum_{i=1}^n$. Inserting \eqref{29}, we get
\begin{align*}
&2(\delta \mathcal N_f)^{\flat}(X)\\
=& -D_{e_i}g(\D_{JX}J)e_i,\g_{\omega}f)+g((\D_{JD_{e_i}X}J)e_i,\g_{\omega}f)+g((\D_{JX}J)\D_{e_i}e_i,\g_{\omega}f) \\
=& D_{e_i}g((\D_{X}J)e_i, \g f)-g((\D_{D_{e_i}X}J)e_i, \g f)-g((\D_{X}J)\D_{e_i}e_i, \g f) \\
=& g((D_{e_i}\D_{X}J)e_i, \g f)-g((\D_{D_{e_i}X}J)e_i, \g f)+g((\D_{X}J)e_i, \D_{e_i}\g f) \\
=& g((D^2_{e_i,X}J)e_i, \g f)+g((\D_{X}J)e_i, \D_{e_i}\g f) \\
=& g((D^2_{e_i,X}J)e_i, \g f)+\Hess f((\D_{X}J)e_i,e_i) .
\end{align*}

 Define $H_f\in \Gamma(\en(TM))$ by $H_f(X)=\D_{X}\g f$, i.e.
$
g(H_f(X), Y)=\Hess f(X,Y).
$
Thus $H_f$ is self-adjoint with respect to $g$, i.e.
\begin{align*}
g(H_f(X),Y)=g(X,H_f(Y)).
\end{align*}
But $\D_XJ$ is  anti-self-adjoint since $J$ is anti-self-adjoint, we see
$$
\sum_{i=1}^n\Hess f((\D_{X}J)e_i,e_i)=\tr (H_f\circ D_XJ)=0.
$$ Then we have
\begin{align*}
(\delta \mathcal N_f)^{\flat}(X)=&\frac{1}{2}g((D^2_{e_i,X}J)e_i, \g f).
\end{align*}
We apply \lemref{aaaa} to conclude
\begin{align*}
(\delta \mathcal N_f)^{\flat}(X)=&\frac{1}{2}g((D^2_{e_i,X}J)e_i-(D^2_{X,e_i}J)e_i, \g f)\\
=&-\frac{1}{2}g((R(e_i,X)J)e_i, \g f)\\
=&-\frac{1}{2}g( R(e_i,X)(Je_i)-JR(e_i,X)e_i, \g f)
\end{align*}
Using (\ref{lich13}) again, we have
\begin{align*}
(\delta \mathcal N_f)^{\flat}(X)=&-\frac{1}{4} R(e_i,Je_i,X,\g f)-\frac{1}{2}\ric(X,\g_{\omega}f)\\
=&-\frac{1}{2}\rho^*(X,\g f)-\frac{1}{2}\ric(X,\g_{\omega}f).
\end{align*}
\end{proof}
We now complete the proof of \thmref{5}, it follows from \lemref{l1} and  \lemref{l2} that 
\begin{align*}
&J(\delta(J\mathcal L_{\g_{\omega}f}J))^{\flat}(X)\\
=&(\delta(\mathcal L_{\g f} J))^{\flat}(JX)+4(\delta \mathcal N_f)^{\flat}(JX)\\
=&\Delta df(X)
-2\ric(\g f, X)-2\ric(JX,J\g f)+2\delta(\D^{+}df-\D^{-}df)(X)\\
=&\Delta d f(X)-4\ric^+(\g f, X)+2\delta(\D^{+}df-\D^{-}df)(X)\\
=&\Delta d f(X)-4(\iota_{\g f}\ric^+)(X)+2\delta(\D^{+}df-\D^{-}df)(X).
\end{align*}
Since 
\begin{align*}
\delta(\iota_{\g f}\ric^+)=&-\D_{e_i}(\iota_{\g f}\ric^+)(e_i)\\
=&-\D_{e_i}(\ric^+(\g f, e_i))+\ric^+(\g f, \D_{e_i}e_i)\\
=&-(\D_{e_i}\ric^+)(\g f, e_i)+\ric^+(\D_{e_i}\g f, e_i)\\
=&\delta\ric^+(\g f)-(\ric^+,\Hess f),
\end{align*} 
we obtain
\begin{align*}
\delta(J(\delta(J\mathcal L_{\g_{\omega}f}J))^{\flat})=&\Delta^2f-4g(\delta\ric^+, df)+4g(\ric^+, \Hess f)+2\delta\delta(\D^{+}df-\D^{-}df)\\
=&\Delta^2f-4g(\delta\ric^+, df)+4g(\rho,dd^c f)+2\delta\delta(\D^{+}df-\D^{-}df).
\end{align*}

\bibliographystyle{plain}
%%
% requires a BiBTeX file sample.bib
\bibliography{ric}

\end{document}